\documentclass[a4paper,12pt]{amsart}
\usepackage[utf8]{inputenc}
\usepackage[T1]{fontenc}
\usepackage[UKenglish]{babel}
\usepackage[a4paper,margin=28mm]{geometry}
\usepackage{verbatim}

\allowdisplaybreaks[4]
\usepackage{times}
\usepackage{dsfont,mathrsfs}
\usepackage{amsmath}
\usepackage{amsthm}
\usepackage{amssymb}
\usepackage{amsfonts}
\usepackage{latexsym}
\usepackage{booktabs}

\usepackage[colorlinks,linkcolor=red,pagebackref]{hyperref}
\usepackage{xcolor}

\usepackage{ulem}

\newtheorem{theorem}{Theorem}[section]
\newtheorem{lemma}[theorem]{Lemma}
\newtheorem{proposition}[theorem]{Proposition}
\newtheorem{corollary}[theorem]{Corollary}
\newtheorem{assumption}[theorem]{Assumption}

\theoremstyle{definition}
\newtheorem{definition}[theorem]{Definition}

\newtheorem{remark}[theorem]{Remark}
\numberwithin{equation}{section}
\renewcommand{\labelenumi}{\roman{enumi})}
\renewcommand\theenumi\labelenumi
\renewcommand{\leq}{\leqslant}
\renewcommand{\le}{\leqslant}
\renewcommand{\geq}{\geqslant}

\newcommand{\tl}{\tilde}
\newcommand{\Be}{\begin{equation}}
\newcommand{\Ees}{\end{equation*}}
\newcommand{\Bes}{\begin{equation*}}
\newcommand{\Ee}{\end{equation}}

\newcommand{\R}{\mathbb{R}}
\newcommand{\E}{\mathbb{E}}
\newcommand{\e}{\varepsilon}

\newcommand{\PP}{\mathbb{P}}

\newcommand{\dif}{\mathrm{d}}

\usepackage{marginnote}
\begin{document}

\title[Approximation of the ergodic measure of SDEs with singular drift by EM]
{
Approximation of the ergodic measure of SDEs with singular drift by Euler-Maruyama scheme
}

\author[X. Jin]{Xinghu Jin}
\address{Xinghu Jin: School of Mathematics,
Hefei University of Technology,
Hefei, Anhui, China;}
\email{2022800009@hfut.edu.cn}

\author[W. Wang]{Wei Wang}
\address{Wei Wang: School of Mathematics, University of Science and Technology of China, Hefei, Anhui, China;}
\email{ww12358@mail.ustc.edu.cn}

\author[L. Xu]{Lihu Xu}
\address{Lihu Xu: 1. Department of Mathematics,
Faculty of Science and Technology,
University of Macau,
Av. Padre Tom\'{a}s Pereira, Taipa
Macau, China; \ \ 2. UM Zhuhai Research Institute, Zhuhai, China.}
\email{lihuxu@umac.mo}

\author[T. Zhang]{Tusheng Zhang}
\address{Tusheng Zhang: School of Mathematics, University of Science and Technology of China, Hefei, Anhui, China.}
\email{tushengz@ustc.edu.cn, tusheng.zhang@manchester.ac.uk}

\keywords{Euler-Maruyama scheme; Singular drift; Zvonkin's transform; Schauder estimate; Invariant measure; Wasserstein-1 distance; Poisson equation}
\subjclass[2010]{}

 \begin{abstract}
 We study the approximation of the ergodic measure of the following stochastic differential equation (SDE) on $\mathbb{R}^d$:
\begin{eqnarray}\label{e:SDEE}
\dif X_t &=& (b_1(X_t)+b_2(X_t)) \dif t+\sigma(X_t) \dif W_t,
\end{eqnarray}
where  $W_t$ is a $d$-dimensional standard Brownian motion, and $b_1: \mathbb{R}^d \mapsto \mathbb{R}^d$, $b_2: \mathbb{R}^d \mapsto \mathbb{R}^d$ and $\sigma: \mathbb{R}^d \mapsto \mathbb{R}^{d\times d}$  are the functions to be specified in Assumption \ref{assump-1} below. In particular, $b_1$ satisfies $b_1\in \mathbb{L}^\infty(\R^d)\cap \mathbb{L}^1(\R^d)$ or $b_1 \in \mathcal{C}_b^{\alpha}(\mathbb{R}^d)$ with  $\alpha\in (0,1)$, which makes
the standard numerical schemes not work or fail to give a good convergence rate.

In order to overcome these two difficulties, we first apply a Zvonkin's transform to SDE \eqref{e:SDEE} and obtain a new SDE which has coefficients with nice properties and admits a unique ergodic measure $\widehat \mu$, then discretize the new equation by Euler-Maruyama scheme to approximate $\widehat \mu$, and finally use the inverse Zvonkin's transform to get an approximation of the ergodic measure of SDE \eqref{e:SDEE}, denoted by $\mu$. Our approximation method is inspired by Xie and Zhang \cite{xie2017ergodicity}.

The proof of our main result is based on the method of introducing a stationary Markov chain, a key ingredient in this method is establishing the regularity of a Poisson equation, which is done by combining the classical PDE local regularity and a nice extension trick introduced by Gurvich \cite{gurvich2014diffusion}.
 \end{abstract}

\maketitle


\section{Introducation} \label{sec:intro}
We are concerned with the following stochastic differential equation (SDE) on $\mathbb{R}^d$:
\begin{eqnarray}\label{e:SDE}
\dif X_t &=& (b_1(X_t)+b_2(X_t)) \dif t+\sigma(X_t) \dif W_t,
\end{eqnarray}
where  $W_t$ is a $d$-dimensional standard Brownian motion, and $b_1: \mathbb{R}^d \mapsto \mathbb{R}^d$, $b_2: \mathbb{R}^d \mapsto \mathbb{R}^d$ and $\sigma: \mathbb{R}^d \mapsto \mathbb{R}^{d\times d}$  are the functions to be specified in Assumption \ref{assump-1} below. This assumption guarantees that SDE \eqref{e:SDE} admits a unique ergodic measure, denoted by $\mu$.  We will propose a numerical scheme for approximating $\mu$, in particular when $b_1$ is singular so that standard numerical schemes will not work or fail to give a good convergence rate.

Euler-Maruyama (EM) scheme is a popular method for numerically solving SDE and has been intensively studied in the past several decades. Most of the known results are about the convergence of EM schemes in finite time interval, see \cite{Bao2019convergence,Fang2016adaptive,Fang2020adaptive,Higham2002strong,yin2010approximation,yuan2008a} and references therein, there are not many papers for studying this convergence when the time tends to infinity. Durmus and Moulines studied the unadjusted Langevin sampling in their celebrated work \cite{Durmus2017Nonasymptotic}, where the drift term $b$ is in a gradient form, and obtained a upper bound for the distance between the ergodic measures of the sampling and the limiting SDE.  More recently, Pages and Panloup \cite{Pages2020Unajusted} gave the error bound in the total variation distance under the similar conditions. If the drift $b$ is locally Lipschitz, Mattingly et al. \cite{Mattingly2002ergodicity} obtained the convergence rate of its EM approximation for invariant measure under a certain distance. \cite{huang2018the,yang2020The} studied the strong convergence of EM scheme in a finite time interval when $b$ is H\"{o}ler continuous, in which they  used the Zvonkin's transform.

In this paper, we assume that the drift term $b:=b_1+b_2$ satisfies one of the following two conditions: (i) $b_1 \in \mathbb{L}^\infty(\mathbb{R}^d)\cap \mathbb{L}^1(\mathbb{R}^d)$, (ii) $b_1 \in \mathcal{C}_b^{\alpha}(\mathbb{R}^d)$ with $\alpha \in (0,1)$. For the case (i), because the value of $b(x)$ for a given $x \in \R^d$ makes no sense, the EM scheme of SDE \eqref{e:SDE} usually fails. For the case (ii), the corresponding EM scheme will converge very slowly in particular as $\alpha$ is small.

In order to overcome the aforementioned two difficulties, we first apply a Zvonkin's transform, denoted by $\Phi$, to SDE \eqref{e:SDE} and obtain a new SDE which has coefficients with nice properties and admits a unique ergodic measure $\widehat \mu$, then discretize the new equation by EM scheme to approximate $\widehat \mu$, and finally
use the inverse Zvonkin's transform $\Phi^{-1}$ to get an approximation of $\mu$. This new scheme is stimulated by the pioneering work by Xie and Zhang \cite{xie2017ergodicity}. We shall show that our new EM scheme performs well and provides a nearly optimal convergence rate.

Let us briefly describe the details for the proof of the main theorem, Theorem \ref{thm:main} below.  Under Assumption \ref{assump-1} below, following the argument in Xie and Zhang \cite{xie2017ergodicity}, we know the new SDE \eqref{e:SDE-1} is ergodic and $\mu=\widehat \mu \circ \Phi$.
As $b_1 \in \mathbb{L}^\infty(\mathbb{R}^d)\cap \mathbb{L}^1(\mathbb{R}^d)$, the new equation has a drift term which is $\gamma$-H\"{o}lder continuous for any $0<\gamma<1$. By the method of introducing a stationary Markov chain, we prove that the approximation error is in an order of $\eta^{\gamma/2}$ for any $0<\gamma<1$ ($\eta$ is the step size of the EM scheme).  A key ingredient in this method is establishing the regularity of a Poisson equation, in which we use the classical PDE local regularity results in Gilbarg and Trudinger \cite[Theorem 6.2]{gilbarg1977schauder} and a nice extension trick introduced by Gurvich \cite{gurvich2014diffusion}. As $b_1 \in \mathcal{C}_b^{\alpha}(\mathbb{R}^d)$, one can discretize SDE \eqref{e:SDE} directly but only get an approximation error of the order $\eta^{\alpha/2}$. However, by combining a Zvonkin's transform and the EM scheme, we can make the approximation error improved to be $\eta^{1/2} |\log \eta|$, this improvement is especially significant when $\alpha$ is small.

Zvonkin's transform was first put forward by Zvonkin when he constructed in \cite{zvonkin1974} a strong solution for the SDE with a "bad" drift. Afterwards, Krylov and R\"{o}ckner \cite{krylov2005} used the same method to obtain the existence and uniqueness of strong solutions to stochastic equations with a local $\mathbb{L}^{x}_p \mathbb{L}^{t}_q$-integrability drift $b$ with $d/p +2/q < 1$. Since then, there has been a surge of studying strong solutions for SDEs with singular drifts, see for instance \cite{Wang2016zhang,xie2017ergodicity,xie2016ergodicity,zhangxc2011} and the references therein.



{This paper will be divided into five sections. In Section \ref{sec:main}, we present assumptions on the drift term $b$ and introduce the Zvonkin's transform, EM scheme, and describe the main theorem. And then we give the regularity for Zvonkin's transform and the results for ergodicity in Section \ref{sec-zovero}. In Section \ref{sec:proof-main}, we give the proof of the main results. In Sections \ref{sec:Calpha} and \ref{sec:Jf-Md2}, we prove two propositions about the regularity of Poisson equation by Schauder estimate in PDE theory.
}

\vskip 3mm

Let us finish this section with some notations that will be frequently used later.
Let $\mathcal{C}(\mathbb{R}^d)$ denote the collection of all continuous functions defined on $\mathbb{R}^d$ and $\mathcal{C}^k(\mathbb{R}^d)$ denotes the collection of $k$-th continuously differentiable functions with integers $k\geq 1$. For $f\in \mathcal{C}^2(\mathbb{R}^d)$, we denote by $\nabla f(x)\in \mathbb{R}^d$ and $\nabla^2 f(x)\in \mathbb{R}^{d\times d}$ the gradient and Hessian matrix for function $f$. And $\mathcal{C}_b (\mathbb{R}^d)$ denotes the family of bounded continuous functions.

The $\alpha$-H\"{o}lder and Lipschitz continuous functions will play an important role in studying the regularity of the Poisson equation below.

For $\alpha \in (0,1)$, we say that  $f$ is $\alpha$-H\"{o}lder continuous with exponent $\alpha$ in $\R^d$ if the semi-norm
\begin{eqnarray}\label{e:Calpha}
[f]_{\alpha} := \sup\limits_{x,y\in \R^d, x\neq y} \frac{ |f(x)-f(y)|  }{ |x-y|^{\alpha} }
\end{eqnarray}
is finite.    $\mathcal{C}^{\alpha}(\mathbb{R}^d)$ denotes the space of  functions  whose $\mathcal{C}^{\alpha}$-norm is finite.  For non-negative integers $k$ and $\alpha\in(0,1]$, the H\"{o}lder space $\mathcal{C}^{k,\alpha}(\mathbb{R}^d)$ is defined as the subspace of $\mathcal{C}^{k}(\mathbb{R}^d)$ consisting of functions whose $k$-th order partial derivatives are $\alpha$-H\"{o}lder continuous. And let $\mathcal{C}^{k,\alpha}_b (\mathbb{R}^d)$ be the space containing all the bounded $\mathcal{C}^{k,\alpha}$ functions  on $\mathbb{R}^d$.

The following notations are adopted from Gilbarg and Trudinger \cite[Section 4]{gilbarg1977schauder}. Let $\mathcal{D}$ be an open subset of $\mathbb{R}^d$. For integers $k=0, 1, 2$ and $\alpha \in (0,1]$, denote
\begin{align*}
[f]_{k,0;\mathcal{D}} &= [f]_{k;\mathcal{D}}=\sup_{x \in \mathcal{D}}   |\nabla^k f(x)|, \\
[f]_{k,\alpha;\mathcal{D}} &= \sup_{\substack{x,y \in \mathcal{D}, x\neq y}}  \frac{ |\nabla^k f(x)- \nabla^k f(y)|  }{ |x-y|^{\alpha} }, \\
|f|_{k;\mathcal{D}} &=\sum_{j=0}^{k} [f]_{j,0;\mathcal{D}}, \\
|f|_{k,\alpha;\mathcal{D}} &= |f|_{k;\mathcal{D}}+[f]_{k,\alpha;\mathcal{D}},
\end{align*}
where $|\nabla^k f(x)-\nabla^k f(y)|$ represents Euclidean distance for $k=0,1$, and $|\nabla^2 f(x)-\nabla^2 f(y)|$ represents the Hilbert-Schmidt norm. In particular, for $k=0$,  $\alpha\in (0,1)$, we denote $[f]_{0,\alpha; \mathcal{D}}=[f]_{\alpha;\mathcal{D}}$.

Let $\mathbb{L}^p(\mathbb{R}^d)$ be the space of all Borel functions $f$ on $\mathbb{R}^d$ with $\mathbb{L}^p$-norm
\begin{equation*}
\|f\|_p:=\left(\int_{\mathbb{R}^d} |f(x)|^p \mathrm{d} x\right)^{1/p} <\infty.
\end{equation*}
Denote by $\mathcal{B}_b(\R^d)$  the space of all bounded Borel-measurable functions. For any $f\in \mathcal{B}_b(\R^d)$,  its $\mathbb{L}^\infty$ norm is defined by
$$\|f\|_{\infty}={\rm ess\ sup}_{x\in \mathbb{R}^d} |f(x)|.$$

For $(p,q) \in [1,\infty] \times (0,2]\setminus \{(\infty,1), (\infty,2)\}$, let $\mathbb{H}_p^q : = (I-\Delta)^{-q/2}(\mathbb{L}^p(\R^d))$ denote the usual Bessel potential space with the norm (see, Xie and Zhang \cite[Section 4]{xie2017ergodicity})
\begin{eqnarray*}
\| f \|_{q,p} &=& \| (I-\Delta) ^{q/2} f  \|_{p}
\ \ \asymp \ \   \| f \|_p + \| \Delta^{q/2} f \|_p,
\end{eqnarray*}
where $x\asymp y$, for $x, y \in \R$, means that there exist some positive constants $c$ and $C$ such that $cx \leq y \leq Cx$, $(I-\Delta)^{q/2}f$ and $\Delta^{q/2}f$ are defined through the Fourier transformation
\begin{equation*}
(I-\Delta)^{q/2}f:=\mathcal{F}^{-1}\left((1+|\cdot|^2)^{q/2} \mathcal{F}f\right), \quad \Delta^{q/2} f :=\mathcal{F}^{-1} (|\cdot|^{q} \mathcal{F}f),
\end{equation*}
 where $\mathcal F$ denotes the Fourier transform.
In particular for  $\mathbb{H}_p^2$, an equivalent norm is defined by
\begin{eqnarray*}
\| f \|_{2,p}= \| f  \|_{p}  + \| \nabla^2 f \|_{p}.
\end{eqnarray*}

Since we need to consider the distance between two probability measures, we recall the Wasserstein-1 distance between two probability measures $\mu_1$ and $\mu_2$ defined as follows (see Hairer and Mattingly \cite[p. 2056]{hairer2008spectral}),
\begin{eqnarray}\label{e:dW}
d_{W}(\mu_1,\mu_2)
&=& \sup_{h \in {\rm Lip(1)}}\left\{\int h(x) \mu_1 (\dif x) - \int h(x) \mu_2 (\dif x) \right \} \nonumber   \\
&=& \sup_{h \in {\rm Lip_0(1)}} \left \{\int h(x) \mu_1 (\dif x) - \int h(x) \mu_2 (\dif x) \right \},
\end{eqnarray}
where ${\rm Lip(1)}$ is the set of Lipschitz functions with Lipschitz constant $1$, that is,  ${\rm Lip(1)}=\{h(\cdot): |h(x)-h(y)|\leq |x-y|$ for all $x,y \in \R^d \}$, and ${\rm Lip_0(1)}:=\{h(\cdot) \in {\rm Lip(1)}: h(0)=0\}$. In addition, for a probability measure $\nu$ and a function $f$, we denote $\nu(f)=\int f(x) \nu(\dif x)$.

For any matrix $A,B\in \mathbb{R}^{d\times d}$, we define the Hilbert-Schmidt inner product as $\langle A, B\rangle_{{\rm HS}}:=\sum_{i,j=1}^{d} A_{ij} B_{ij}$. Given a matrix $A\in \mathbb{R}^{d\times d}$, its Hilbert-Schmidt norm is $\|A\|_{{\rm HS}}=\sqrt{\sum_{i,j=1}^{d} A^2_{ij}}$. For matrixes $A$ and $B$, $A\leq B$ means $B-A$ is positive definite, and ${\rm I_d}$ means the $d$-dimensional identity matrix. $A^{\prime}$ means the transpose of matrix $A$.

For any real number $R>0$ and $x\in \R^d$, we denote the open ball with radius $R$ and center $x$ in $\mathbb{R}^d$ as follows:
\begin{eqnarray*}
B_R(x) &=& \{ z\in \mathbb{R}^d: |z-x| < R   \}.
\end{eqnarray*}

\section{Main results}\label{sec:main}

Throughout this paper, we impose the following assumptions:
\begin{assumption}\label{assump-1}
({\bf A1}) The drift term $b(x)$ has the following form
\begin{eqnarray}\label{e:b}
b(x) &=& b_1(x)+b_2(x),
\end{eqnarray}
where $b_2$ is such that there are  some positive constants $\theta_1,\theta_2,\theta_3>0$ satisfying
\begin{eqnarray}
\langle x,b_2(x)\rangle &\leq& -\theta_1 |x|^{2}+\theta_2,\label{e:dissi-b2} \\
|b_2(x)-b_2(y)| &\leq&    \theta_3 |x-y|, \label{e:lingro-b2}
\end{eqnarray}
for all $x, y \in \R^d$. $b_1$ is called singular part and satisfies {\bf one} of the following two conditions:

{\bf Case 1:}  $b_1\in \mathbb{L}^\infty(\R^d)\cap \mathbb{L}^1(\R^d)$,

{\bf Case 2:} $b_1 \in \mathcal{C}_b^{\alpha}(\mathbb{R}^d)$ with  $\alpha\in (0,1)$.
\vskip 3mm
\noindent ({\bf A2}) The diffusion matrix $\sigma$ is Lipschitz  and there is a positive constant $\lambda_{\sigma}\in (0,1)$ such that
\begin{eqnarray}\label{e:sigma-ell}
\lambda_{\sigma} {\rm I}_d  \ \ \leq \ \  \sigma(x)\sigma^{\prime}(x) \ \  \leq \ \ \lambda_{\sigma}^{-1} {\rm I}_d.
\end{eqnarray}
\end{assumption}

\begin{remark}\label{remark-Lp}
Assumption {\bf (A1) Case 1} implies $b_1\in \mathbb{L}^p(\mathbb{R}^d)$ for any $1<p<\infty$.
\end{remark}

Due to the singularity of $b_1$, we will use the well known Zvonkin's transform to study the SDE \eqref{e:SDE}.

\subsection{Zvonkin's transform}
We consider the following elliptic equation: for $\lambda>0$
\begin{eqnarray}\label{e:pde-b1}
(\mathcal{L}_2^{\sigma}-\lambda) u + \mathcal{L}_1^{b_1} u &=& -b_1,
\end{eqnarray}
where \begin{eqnarray*}
\mathcal{L}_1^{b_1} u(x) \ \ = \ \ \langle b_1(x), \nabla u(x) \rangle,
\quad
\mathcal{L}_2^{\sigma} u(x) \ \ = \ \  \frac{1}{2} \langle \sigma(x) \sigma^{\prime}(x), \nabla^2 u(x) \rangle_{\rm HS}.
\end{eqnarray*}
Under Assumption \ref{assump-1}, the Eq. \eqref{e:pde-b1} admits a unique solution with certain regularity, see Lemmas \ref{lem:regu-ul} and \ref{lem:regu-u} below. Define
\begin{eqnarray}\label{e:Phi}
\Phi(x) &=& x+ u(x),
\end{eqnarray}
and $\Phi: \R^d \rightarrow \R^d$ is called {Zvonkin's transform}. For more details about Zvonkin's transform, we refer the reader to \cite{flandoli2010,krylov2005,xie2016ergodicity,zhangxc2011} and the references therein.

Thanks to Lemmas \ref{lem:regu-ul} and \ref{lem:regu-u} below,
we can apply  Zvonkin's transform to the SDE \eqref{e:SDE} and immediately obtain the following lemma by applying the It\^{o}'s formula.
\begin{lemma}\label{lem:Y}
$X_t$ solves SDE \eqref{e:SDE} if and only if $Y_t:=\Phi(X_t)$ solves
\begin{eqnarray}\label{e:SDE-1}
\dif Y_t &=& \widehat{b}(Y_t) \dif t + \widehat{\sigma}(Y_t) \dif W_t
\end{eqnarray}
with initial value $Y_0=\Phi(X_0)$ and
\begin{eqnarray}\label{e:bsigma}
\widehat{b}(y) \ = \ (\lambda u+\nabla \Phi \cdot b_2) \circ \Phi^{-1}(y), \quad
\widehat{\sigma}(y) \ = \ (\nabla \Phi \cdot \sigma) \circ \Phi^{-1}(y), \forall y\in\R^d.
\end{eqnarray}
\end{lemma}
Under Assumption \ref{assump-1}, we can show that the processes $(Y_t)_{t\geq 0}$ in \eqref{e:SDE-1} and $(X_t)_{t\geq 0}$ in $\eqref{e:SDE}$ are both exponentially ergodic, denoting their ergodic measures by $\widehat{\mu}$ and $\mu$ respectively, see Lemmas \ref{pro:ergodic-Y} and \ref{pro:ergodic-X} in Section \ref{sec-zovero}.

\subsection{\bf EM scheme and main results}
We aim to develop a numerical scheme to approximate the ergodic measure $\mu$ of the process $(X_t)_{t\geq 0}$.
Due to the relation
$$\mu=\widehat{\mu} \circ \Phi,$$
which is established in Xie and Zhang \cite[Proposition 2.8]{xie2017ergodicity}, it is natural to first develop a numerical scheme to approximate $\widehat{\mu}$ and then transform the approximation by $\Phi$.
To this end, let us first consider the EM scheme for SDE \eqref{e:SDE-1}.

Let $Z_0=Y_0$, the EM scheme for the SDE \eqref{e:SDE-1} reads as
\begin{eqnarray}\label{e:MC-R}
Z_{k+1} &=& Z_{k}+ \eta \widehat{b}(Z_k)+  \sqrt{\eta} \widehat{\sigma}(Z_k)\xi_{k+1}, \quad \forall  k\in \mathbb{N}_0,
\end{eqnarray}
where $\eta>0$ is the step size, $(\xi_k)_{k\in \mathbb{N}}$ are independent standard Gaussian random variables. For an integer $k$, $\xi_k$ is independent of $Z_0,Z_1,\cdots,Z_{k-1}$.
It is easy to see that $(Z_k)_{k\in \mathbb{N}_0}$ is a Markov chain.  Under Assumption \ref{assump-1}, we shall show that $(Z_k)_{k\in \mathbb{N}_0}$ is exponentially ergodic, denote its ergodic measure by $\widehat{\mu}_{\eta}$, see Lemma  \ref{pro:ergodic-MC} below.

We shall use the measure $\widehat{\mu}_{\eta}  \circ \Phi$ to approximate  $\mu$ and derive an error bound in Wasserstein-1 distance, this is precisely stated in the main theorem as follows.


\begin{theorem}\label{thm:main}
(i) Let Assumption {\bf (A1) Case 1} and {\bf (A2)} hold, then for any  $\gamma\in (0,1)$, there exists some positive constant $C$ depending on $\gamma$ such that
\begin{eqnarray*}
d_W(\mu,\widehat{\mu}_{\eta} \circ \Phi)   &\leq&  C \eta^{\frac{\gamma}{2}}.
\end{eqnarray*}
Furthermore, for any given error $\varepsilon>0$, taking $\eta \asymp \varepsilon^{\frac{2}{\gamma}}$ and $k\asymp \varepsilon^{-\frac{2}{\gamma}}|\log \varepsilon|$, we know
\begin{eqnarray*}
d_W(\mathcal{L}(\Phi^{-1}(Z_k)),\mu) &\leq& \varepsilon,
\end{eqnarray*}
where $\mathcal{L}(\Phi^{-1}(Z_k))$ is the law of $\Phi^{-1}(Z_k)$.


(ii) Let Assumption {\bf (A1) Case 2} and {\bf (A2)}  hold, then there exists some positive constant $C$ independent of $\eta$ such that
\begin{eqnarray*}
d_W(\mu,\widehat{\mu}_{\eta} \circ \Phi)   &\leq&  C \eta^{\frac{1}{2}} |\log \eta|.
\end{eqnarray*}
Furthermore, for any given error $\varepsilon>0$, taking $\eta \asymp \varepsilon^{\frac{8}{3}}$ and $k\asymp \varepsilon^{-\frac{8}{3}}|\log \varepsilon|$, we know
\begin{eqnarray*}
d_W(\mathcal{L}(\Phi^{-1}(Z_k)),\mu) &\leq& \varepsilon,
\end{eqnarray*}
where $\mathcal{L}(\Phi^{-1}(Z_k))$ is the law of $\Phi^{-1}(Z_k)$.
\end{theorem}

\section{Regularity of Zvonkin's transform and Ergodicity}\label{sec-zovero}

\subsection{Regularity of Zvonkin's transform}

The following Lemma is from
Xie and Zhang \cite[Theorem 7.6, (4.2) and (7.23)]{xie2017ergodicity}.
\begin{lemma}[Regularity, $b_1 \in \mathbb{L}^{\infty}(\R^d)\cap\mathbb{L}^{1}(\R^d)$] \label{lem:regu-ul}
{Let Assumption {\bf (A1) Case 1} and {\bf (A2)} hold. Then, for any $p>d$, there exists some $\lambda_1=\lambda_1(p)$ such that for all $\lambda\geq \lambda_1$, we have a unique solution $u\in \mathbb{H}_{p}^{2}$ to Eq. \eqref{e:pde-b1} and constants $C=C(p,d,\lambda)$, $C'=C'(p,d)$ such that,
\begin{eqnarray*}
\|\nabla^2 u\|_{p} &\leq& C \|b_1\|_{p}, \\
\|u\|_{\infty}+\|\nabla u\|_{\infty} &\leq& C' \lambda^{-\frac{1}{2} \gamma},
\end{eqnarray*}
where $\gamma=1-d/p$.

In particular,  we have $u\in \mathcal{C}_b^{1,\gamma}(\R^d)$  by the Sobolev embedding $\mathbb{H}_p^2\subset \mathcal{C}_b^{1,\gamma}$.
}
\end{lemma}

By the similar method as in Flandoli et al. \cite[Section 2]{flandoli2010}, we have the regularity for $b_1\in \mathcal{C}_b^{\alpha}(\R^d)$ as follows,
\begin{lemma}[Regularity, $b_1\in \mathcal{C}_b^{\alpha}(\R^d)$]\label{lem:regu-u}
Let Assumption {\bf (A1) Case 2} and {\bf (A2)} hold.
For any $z\in \mathbb{R}^d$, let $u\in \mathcal{C}_{b}^{2,\alpha}(\mathbb{R}^d)$ be a classical solution of Eq.  \eqref{e:pde-b1} in $B_1(z)$, there exist positive constants $C=C(\lambda,d, \alpha)$ and $C'=C'(d,\alpha)$ such that, for $x, y\in B_{1/8}(z)$
\begin{eqnarray*}
|\nabla^2 u(x)-\nabla^2 u(y)| &\leq& C (  |b_1|_{0,\alpha;B_1(z)} +|u|_{0;B_1(z)} )|x-y|^\alpha, \\
\|u\|_{\infty}+\|\nabla u\|_{\infty} &\leq& C' \lambda^{-1}.  
\end{eqnarray*}
\end{lemma}

Combining these two lemmas with Lemma \ref{lem:Y}, we have the following corollary:
\begin{corollary}\label{cor:hatb}

(i) Let Assumption {\bf (A1) Case 1} and {\bf (A2)} hold. For $\lambda>0$ large enough, then $\widehat{b},\widehat{\sigma} \in \mathcal{C}^{\gamma}(\R^d)$ with $\gamma$ given in Lemma \ref{lem:regu-ul};

(ii) Let Assumption {\bf (A1) Case 2} and {\bf (A2)} hold. Then, for sufficiently large $\lambda>0$, $\widehat{b},\widehat{\sigma} \in \mathcal{C}^{1,\alpha}(\R^d)$.
\end{corollary}

\begin{proof}
It follows from Lemmas \ref{lem:regu-ul} and \ref{lem:regu-u} that for $\lambda$ large enough, the map $x \mapsto \Phi(x)$ forms a $\mathcal{C}^1$-diffeomorphism and moreover,
\begin{eqnarray}\label{e:phi-infty}
\frac 1 2 \ \  \leq \ \ \|\nabla \Phi\|_{\infty},\|\nabla \Phi^{-1}\|_{\infty} \ \ \leq  \ \ 2.
\end{eqnarray}

(i) As $b_1\in \mathbb{L}^\infty(\R^d)\cap \mathbb{L}^1(\R^d)$, it follows from Lemma \ref{lem:regu-ul} that $u\in \mathcal{C}_b^{1,\gamma}(\R^d)$ with $\gamma=1-d/p$, so $\Phi \in \mathcal{C}^{1,\gamma}(\R^d)$. It is clear that $\widehat{b},\widehat{\sigma} \in \mathcal{C}^{\gamma}(\R^d)$ from expressions of $\widehat{b}$ and $\widehat{\sigma}$ in Eq. \eqref{e:bsigma}.

(ii) As $b_1 \in \mathcal{C}_b^{\alpha}(\mathbb{R}^d)$ with  $\alpha\in (0,1)$, it follows from Lemma \ref{lem:regu-u} that $u\in \mathcal{C}_b^{2,\alpha}(\R^d)$, so $\Phi \in \mathcal{C}^{2,\alpha}(\R^d)$. This implies  that $\widehat{b},\widehat{\sigma} \in \mathcal{C}^{1,\alpha}(\R^d)$ by the  expressions of $\widehat{b}$ and $\widehat{\sigma}$ in Eq. \eqref{e:bsigma}.
\end{proof}

\subsection{Ergodicity}
Let $f:\R^d \rightarrow \R$ be a measurable function.
If $Y_0=y$, we denote the process $(Y_t)_{t\geq 0}$ by $(Y_t^y)_{t\geq 0}$, define
\begin{eqnarray*}
\widehat{P}_t f(y) \ \ = \ \ \mathbb{E} f(Y_t^y), \quad \ t\geq 0,
\end{eqnarray*}
as long as $\mathbb{E}|f(Y_t^y)|<\infty$. Similarly, we introduce the notation $(Z_k^z)_{k\in \mathbb{N}_0}$ and define
 \begin{eqnarray*}
\widehat{Q}_k f(z) \ \ = \ \ \mathbb{E} f(Z_k^z), \quad  k\in \mathbb{N}_0,
\end{eqnarray*}
as long as $\mathbb{E}|f(Z_k^z)|<\infty$.

The following two lemmas are about the ergodicity of the process $(Y_t)_{t\geq 0}$ and the Markov chain $(Z_k)_{k\in \mathbb{N}_0}$, whose proofs are standard and not new. For the completeness, we will give the details in Appendix \ref{app:ergodicity}.
\begin{lemma}\label{pro:ergodic-Y}
Let Assumption \ref{assump-1} hold, the process $(Y_t)_{t\geq 0}$ given by SDE  \eqref{e:SDE-1} is exponentially ergodic with a unique invariant measure $\widehat{\mu}$. More precisely,
\begin{eqnarray*}
\sup_{ |f| \leq 1+|\cdot|^2 }  \{  \widehat{P}_t f(y) - \widehat{\mu}(f)  \} &\leq& C (1+|y|^2) e^{-ct},  \forall \ y \in \R^d.
\end{eqnarray*}
This implies
$$d_W(\mathcal L(Y^y_t), \widehat \mu) \le  C (1+|y|^2) e^{-ct},  \forall \ y \in \R^d,$$
where $\mathcal L(Y^y_t)$ is the law of $Y^y_t$. 
\end{lemma}

\begin{lemma}\label{pro:ergodic-MC}
Let Assumption \ref{assump-1} hold, the Markov chain $(Z_k)_{k\in \mathbb{N}_0}$ given in Eq. \eqref{e:MC-R} is exponential ergodic with a unique invariant  measure $\widehat{\mu}_{\eta}$. More precisely, {there exist some positive constants $c$ and $C$ both independent of $k$ and $\eta$ such that}
\begin{eqnarray*}
\sup_{ |f| \leq 1+|\cdot|^2 }  \{  \widehat{Q}_k f(z) - \widehat{\mu}_{\eta}(f)  \} &\leq& C \eta^{-1} e^{-ck\eta}, \ \ \forall \ z \in \R^d.
\end{eqnarray*}
This implies
$$d_W(\mathcal L(Z^z_k), \widehat \mu_{\eta}) \leq  C \eta^{-1} e^{-ck\eta}, \ \ \forall \ z \in \R^d,$$
where $\mathcal L(Z^z_k)$ is the law of
$Z^z_k$. Furthermore, for any integers $k\geq 1$, {one has $\widehat{\mu}_{\eta}(|\cdot|^k) \leq C$ with some positive constant $C$ independent of $\eta$}.
\end{lemma}
By the relationship of ergodicity between processes $(X_t)_{t\geq 0}$ and $(Y_t)_{t\geq 0}$ (see \cite[Proposition 2.8]{xie2017ergodicity}), we have
\begin{lemma}\label{pro:ergodic-X}
Let Assumption \ref{assump-1} hold, the solution process $(X_t)_{t\geq 0}$ given by  SDE \eqref{e:SDE} is ergodic with an ergodic measure $\mu$. In addition, one has
\begin{eqnarray}\label{e:mus-1}
\mu & = &   \widehat{\mu} \circ \Phi.
\end{eqnarray}
\end{lemma}

\section{Proof of Theorem \ref{thm:main}.} \label{sec:proof-main}
We shall prove the main theorem by introducing a stationary Markov chain associated to the EM scheme, in which the Poisson equation will play an important role. This trick is very similar to that in \cite{fang2019Multivariate}. 
\subsection{Poisson equation}
Let $\widehat{\mathcal{A}}$ be the generator of the process $(Y_t)_{t\geq 0}$ in SDE \eqref{e:SDE-1}, that is,
\begin{eqnarray}\label{e:hatA}
\widehat{\mathcal{A}} f(x) &=& \langle  \widehat{b}(x), \nabla f(x)  \rangle + \frac{1}{2}  \langle \widehat{a}(x), \nabla^2 f(x) \rangle_{\rm HS},
\quad f\in \mathcal{D}(\widehat{\mathcal{A}}),
\end{eqnarray}
where $\widehat{a}(x)=\widehat{\sigma}(x) \widehat{\sigma}^{\prime}(x)$ and $\mathcal{D}(\widehat{\mathcal{A}})$ is the domain of generator $\widehat{\mathcal{A}}$. For any $h\in {\rm Lip}(1)$, we consider the following Poisson equation:
\begin{eqnarray}\label{e:steinequ}
\widehat{\mathcal{A}} f(x)  &=& h(x) - \widehat{\mu}(h).
\end{eqnarray}

Regularities of the solution $f$ in Eq. \eqref{e:steinequ} play crucial roles in proving our main results. Before proving those regularities, we first give the representation for the solution $f$ to Eq. \eqref{e:steinequ} as below.
\begin{lemma}\label{lem:stein-f}
Let Assumption \ref{assump-1} hold, and  $h\in {\rm Lip}(1)$. Then the solution to Eq.  \eqref{e:steinequ} is given by
\begin{eqnarray}\label{e:fh-1}
f(x) &=& -\int_0^{\infty}  \widehat{P}_t [ h(x) - \widehat{\mu}(h) ] \dif t,
\end{eqnarray}
where $\widehat{P}_t$ is the semigroup of the process $(Y_t)_{t\geq 0}$. Furthermore, there exists some positive constant $C$ such that
\begin{eqnarray*}
|f(x)| &\leq& C(1+|x|^2).
\end{eqnarray*}
\end{lemma}

\begin{proof}
Since $h \in {\rm Lip}(1)$, we know $|h(x)|\leq C(1+|x|^2)$ for some positive constant $C$. It follows from Lemma \ref{pro:ergodic-Y} that
\begin{eqnarray*}
\int_{0}^{\infty} \widehat{P}_t[h(x)-\widehat{\mu}(h)] \dif t
\end{eqnarray*}
is well defined and
\begin{eqnarray*}
|f(x)| &\leq& \int_{0}^{\infty} |\widehat{P}_t[h(x)-\widehat{\mu}(h)] | \dif t
\ \ \leq \ \ \int_{0}^{\infty} C(1+|x|^2) e^{-ct}\dif t
\ \ \leq \ \  C(1+|x|^2).
\end{eqnarray*}
The expression for $f$ in Eq.  \eqref{e:fh-1} can be proved similarly as in Fang et al. \cite[Proposition 6.1]{fang2019Multivariate}. We omit the details here.
\end{proof}

\subsection{Proof of Theorem \ref{thm:main} (i)}\
In order to prove Theorem \ref{thm:main} (i), we  need the following proposition whose proof is given  in Section \ref{sec:Calpha}.
\begin{proposition}[Regularities, $b_1\in \mathbb{L}^{\infty}(\mathbb{R}^d) \cap \mathbb{L}^{1}(\mathbb{R}^d) $]\label{pro:regu1}
Let $f$ be the solution to Eq. \eqref{e:steinequ} under Assumption {\bf (A1) Case 1} and {\bf (A2)}. Then, for any $p>d$, there exists some positive constant $C=C(d,p)$ such that
\begin{eqnarray*}
|\nabla f(x)| &\leq& C(1+|x|^{3}), \\
|\nabla^2 f(x)| &\leq& C (1+|x|^{4}), \\
\sup_{y:|y-x|\leq 1/8} \frac{|\nabla^2 f(x)-\nabla^2 f(y)|}{ |x-y|^{\gamma}} &\leq& C(1+|x|^{4+\gamma}),
\end{eqnarray*}
where $\gamma=1-\frac{d}{p}$ is as in Lemma \ref{lem:regu-ul}.
\end{proposition}

\begin{proof}[Proof of Theorem \ref{thm:main} (i).]
 For large enough $\lambda$ in Eq. \eqref{e:pde-b1}, we know that
\begin{eqnarray*}
\frac{1}{2} \leq \|\nabla \Phi \|_{\infty},\|\nabla \Phi^{-1}\|_{\infty} \leq 2.
\end{eqnarray*}
For $h \in {\rm Lip}(1)$, let $\tl{h} = \Phi (h)$. It is clear that  $\tl{h}$ is also Lipschitz continuous with $ \| \nabla \tl{h} \|_{\infty} \leq 2$.  Therefore, we have the following relationship between
$d_W(\mu ,\widehat{\mu}_{\eta}  \circ \Phi)$
and  $d_W(\widehat{\mu},\widehat{\mu}_{\eta} )$:
\begin{eqnarray*}
d_W(\mu,\widehat{\mu}_{\eta} \circ \Phi)  &=& d_W(\widehat{\mu} \circ \Phi ,\widehat{\mu}_{\eta} \circ \Phi)
\ \ = \ \  \sup_{h\in {\rm Lip(1)} } \{  \widehat{\mu}  \circ \Phi (h) - \widehat{\mu}_{\eta}  \circ \Phi (h) \} \\
&=& \sup_{h\in {\rm Lip(1)} } \{  \widehat{\mu}(\tl{h}) - \widehat{\mu}_{\eta}(\tl{h}) \}
\ \ = \ \ \sup_{h\in {\rm Lip(1)} } \left\{  \| \nabla \tl{h} \|_{\infty}  [\widehat{\mu}(\frac{\tl{h}}{\| \nabla \tl{h} \|_{\infty}}) - \widehat{\mu}_{\eta}(\frac{\tl{h}}{\| \nabla \tl{h} \|_{\infty}})] \right\} \\
&\leq& 2 \sup_{h\in {\rm Lip(1)} } \{  \widehat{\mu}(h) - \widehat{\mu}_{\eta}(h) \} \\
&=& 2d_W(\widehat{\mu}, \widehat{\mu}_{\eta} ).
\end{eqnarray*}

To get the estimate for $d_W(\mu,\widehat{\mu}_{\eta} \circ \Phi)$, it suffices to bound $d_W(\widehat{\mu}, \widehat{\mu}_{\eta})$. Assume  the law of the initial value $Z_0$ is the invariant measure $\widehat{\mu}_{\eta}$. Then, $(Z_k)_{k\in \mathbb{N}_0}$ is a stationary Markov chain. Let $\xi_1 \sim \mathcal{N}(0,{\rm I}_d)$, denote
\begin{eqnarray}\label{e:delta}
\delta &=& Z_1 - Z_0
\ \ = \ \ \eta \widehat{b}(Z_0) + \eta^{\frac{1}{2}} \widehat{\sigma}(Z_0) \xi_1.
\end{eqnarray}

Let $f$ be the solution to Eq. \eqref{e:steinequ}. Since $Z_1$ and $Z_0$ have the same distribution, we have
\begin{eqnarray*}
0&=& \mathbb{E}f (Z_1)-\mathbb{E}f(Z_0)  \\
&=& \mathbb{E}[ \langle \delta, \nabla f(Z_0) \rangle  ]
+\mathbb{E} \int_0^1 \int_0^1 r \langle \delta \delta^{\prime}, \nabla^2 f (Z_0+\tl{r} r \delta ) \rangle_{\textrm{HS}} \dif \tl{r} \dif r   \\
&=&\mathbb{E}[  \langle \widehat{b}(Z_0),\nabla f(Z_0) \rangle ]\eta+\mathbb{E} \int_0^1 \int_0^1 r \langle \delta \delta^{\prime}, \nabla^2 f(Z_0+\tl{r} r \delta ) \rangle_{\textrm{HS}} \dif \tl{r} \dif r,
\end{eqnarray*}
where
\begin{eqnarray*}
\mathbb{E}[ \langle \delta, \nabla f(Z_0) \rangle  ]
&=& \mathbb{E}[ \langle \mathbb{E}(\delta|Z_0),\nabla f(Z_0)\rangle  ]
\ = \ \mathbb{E}[  \langle \widehat{b}(Z_0),\nabla f(Z_0) \rangle ]\eta.
\end{eqnarray*}
In addition,
\begin{eqnarray*}
&& \mathbb{E} \int_0^1 \int_0^1 r \langle \delta \delta^{\prime}, \nabla^2 f (Z_0+\tl{r} r \delta ) \rangle_{\textrm{HS}} \dif \tl{r} \dif r  \\
&=& \frac{1}{2}\mathbb{E} \langle \delta \delta^{\prime}, \nabla^2 f(Z_0) \rangle_{\textrm{HS}}+
\mathbb{E} \int_0^1 \int_0^1 r \langle \delta \delta^{\prime}, \nabla^2 f(Z_0+\tl{r} r \delta )-\nabla^2 f(Z_0) \rangle_{\textrm{HS}} \dif \tl{r} \dif r  \\
&=& \frac{\eta}{2} \mathbb{E} [ \langle \sigma(Z_0) \sigma^{\prime}(Z_0), \nabla^2 f(Z_0) \rangle_{\textrm{HS}}  ] +\frac{\eta^2}{2} \mathbb{E} [  \langle \widehat{b}(Z_0) \widehat{b}^{\prime} (Z_0), \nabla^2 f(Z_0) \rangle_{\textrm{HS}} ]   \\
&\ & \ \ \ \ \ \ \  +\mathbb{E} \int_0^1 \int_0^1 r \langle \delta \delta^{\prime}, \nabla^2 f (Z_0+\tl{r} r \delta )-\nabla^2 f(Z_0) \rangle_{\textrm{HS}} \dif \tl{r} \dif r.
\end{eqnarray*}

Collecting the terms above, we obtain
\begin{eqnarray*}
\mathbb{E}[\widehat{\mathcal{A}}f(Z_0) ]&=&\frac \eta 2 {\rm I}+\frac{1}{\eta} {\rm II},
\end{eqnarray*}
where
\begin{eqnarray*}
{\rm I} &=& -\mathbb{E} [ \langle \widehat{b}(Z_0) \widehat{b}^{\prime}(Z_0), \nabla^2 f(Z_0) \rangle_{\textrm{HS}} ], \\
{\rm II} &=& -\mathbb{E} \int_0^1 \int_0^1 r \langle \delta \delta^{\prime}, \nabla^2 f(Z_0+\tl{r} r \delta )-\nabla^2 f(Z_0) \rangle_{\textrm{HS}} \dif\tl{r} \dif r.
\end{eqnarray*}

Under Assumption {\bf (A1) Case 1} and {\bf (A2)}, we claim that there exists constant $C>0$ independent of $\eta$ such that
\begin{eqnarray}
|{\rm I}| &\leq& C, \label{e:claim-1}   \\
|{\rm II}| &\leq& C\eta^{1+\frac{\gamma}{2}}, \label{e:claim-2}
\end{eqnarray}
which implies that
\begin{eqnarray*}
\sup_{h \in {\rm Lip}(1)} |\mathbb{E}[\widehat{\mathcal{A}}f(Z_0)]|
\ \ \leq \ \  \sup_{h\in {\rm Lip}(1)} \{ \frac{\eta}{2} |{\rm I}|+\frac{1}{\eta} |{\rm II}| \}
\ \  \leq \ \ C\eta^{\frac{\gamma}{2}}.
\end{eqnarray*}
Combining this with Eq. \eqref{e:steinequ},  there exists constant $C>0$ independent of $\eta$ such that
\begin{eqnarray*}
d_W(\widehat{\mu},\widehat{\mu}_{\eta})
\ \ = \ \  \sup_{h\in {\rm Lip(1)} } \{  \widehat{\mu}_{\eta} (h) - \widehat{\mu}(h) \}
\ \ = \ \  \sup_{h\in {\rm Lip(1)} } \{ \mathbb{E}^{\widehat{\mu}_{\eta}}[\widehat{\mathcal{A}}f(Z_0)]  \}
\ \ \leq \ \ C\eta^{\frac{\gamma}{2}}.
\end{eqnarray*}

Now we show claims \eqref{e:claim-1} and \eqref{e:claim-2}. By the estimate for $|\nabla^2 f|$ in Proposition \ref{pro:regu1} and the linear growth of $\widehat{b}$ in Lemma \ref{lem:tlb}, there exists constant $C>0$ independent of $\eta$ such that
\begin{eqnarray*}
|{\rm I}|
&\leq& \mathbb{E} [ |\langle \widehat{b}(Z_0) \widehat{b}^{\prime}(Z_0), \nabla^2 f(Z_0) \rangle_{\textrm{HS}}| ]
\ \ \leq \ \ C\E(1+|Z_0|^{6})
\ \ \leq \ \ C,
\end{eqnarray*}
where we have used the fact that  $\widehat{\mu}_{\eta}(|\cdot|^{6}) \leq C$ {and the constant $C$ is independent of $\eta$} (see Lemma \ref{pro:ergodic-MC}). The  claim \eqref{e:claim-1} is proved.

To prove the  claim \eqref{e:claim-2}, we write
\begin{eqnarray}\label{e:II-12}
{\rm II}
&=& -\mathbb{E} \int_0^1 \int_0^1 r \langle \delta \delta^{\prime}, \nabla^2 f(Z_0+\tl{r} r \delta )-\nabla^2 f(Z_0) \rangle_{\textrm{HS}} 1_{ \{ | \delta |\leq 1/8 \} } \dif\tl{r} \dif r  \nonumber  \\
&&-\mathbb{E} \int_0^1 \int_0^1 r \langle \delta \delta^{\prime}, \nabla^2 f(Z_0+\tl{r} r \delta )-\nabla^2 f(Z_0) \rangle_{\textrm{HS}} 1_{ \{ | \delta |> 1/8 \} } \dif\tl{r} \dif r.
\end{eqnarray}

Using the third estimate in Proposition \ref{pro:regu1}, there exists some positive constant $C$ independent of $\eta$ satisfying
\begin{eqnarray}\label{e:II-2}
&& \left| \mathbb{E} \int_0^1 \int_0^1 r \langle \delta \delta^{\prime}, \nabla^2 f(Z_0+\tl{r} r \delta )-\nabla^2 f(Z_0) \rangle_{\textrm{HS}} 1_{ \{ | \delta |\leq 1/8 \} } \dif\tl{r} \dif r  \right|  \nonumber \\
&=&  \left| \mathbb{E} \int_0^1 \int_0^1 r \langle \delta \delta^{\prime},  \frac{ \nabla^2 f(Z_0+\tl{r} r \delta )-\nabla^2 f(Z_0)} {  |\tl{r} r  \delta|^{\gamma}}
 |  \tl{r} r  \delta|^{\gamma} \rangle_{\textrm{HS}} 1_{ \{ | \delta |\leq 1/8 \} } \dif\tl{r} \dif r     \right|  \nonumber \\
&\leq& C \E [ |\delta|^{2+\gamma}  (1+ |Z_0|^{4+\gamma} )  ] \nonumber \\
&\leq& C\eta^{1+\frac{\gamma}{2}},
\end{eqnarray}
where the last inequality holds due to \eqref{e:delta} and the fact that  $\widehat{\mu}_{\eta}(|\cdot|^{6+2\gamma}) \leq C$ {and the constant $C$ is independent of $\eta$} (see Lemma  \ref{pro:ergodic-MC}).

Using the second estimate in Proposition \ref{pro:regu1}, there exists some positive constant $C$ independent of $\eta$ such that
\begin{eqnarray}\label{e:II-1}
&& |\mathbb{E} \int_0^1 \int_0^1 r \langle \delta \delta^{\prime}, \nabla^2 f(Z_0+\tl{r} r \delta )-\nabla^2 f(Z_0) \rangle_{\textrm{HS}} 1_{ \{ |  \delta |> 1/8 \} } \dif\tl{r} \dif r| \nonumber \\
&\leq& C \mathbb{E} \int_0^1 \int_0^1  |\delta|^2 [ |\nabla^2 f(Z_0+\tl{r} r \delta )| + |\nabla^2 f(Z_0)|  ] 1_{ \{ |  \delta |> 1/8 \} } \dif\tl{r} \dif r \nonumber \\
&\leq& C\E[  |\delta|^2 ( 1+ |Z_0|^{4} + |\delta|^{4}  ) 1_{ \{ |  \delta |> 1/8 \} }   ]  \nonumber \\
&\leq& C\left(\E[  |\delta|^6 ]\right)^{\frac 13} \left(\E[ ( 1+ |Z_0|^{4} + |\delta|^{4}  )^3]\right)^{\frac 13}  \left(\PP(|\delta|\geq 1/8)\right)^{\frac 13} \nonumber \\
&\leq& C\eta \left(\PP(|\delta|\geq 1/8)\right)^{\frac 13} \nonumber \\
&\leq& C\eta^2
\end{eqnarray}
where the third inequality holds due to the Chebyshev's inequality and the fact that $\widehat{\mu}_{\eta}(|\cdot|^{12}) \leq C$  ($C$ is independent of $\eta$) by Lemma \ref{pro:ergodic-MC}), and the last inequality is by the following observation:
\begin{equation}
\begin{split}
\PP( |\delta|\geq 1/8) & \le \PP\left(\eta |\hat b(Z_0)| \le \frac1{16}\right)+\PP\left(\sqrt{\eta}|\hat \sigma(Z_0) \xi_1| \le \frac1{16}\right) \\
& \le \PP\left(|Z_0| \le \frac{c}{16\eta}\right)+\PP\left(| \xi_1| \le \frac{c}{16 \sqrt{\eta}}\right) \le \frac{C \E|Z_0|^3}{\eta^3}+Ce^{-\frac{1}{C\eta}}.
\end{split}
\end{equation}
Combining \eqref{e:II-12}, \eqref{e:II-2} and \eqref{e:II-1}, we prove the claim \eqref{e:claim-2}.

Next we prove the second part of (i).  Noticing
\begin{eqnarray*}
d_W(\mathcal{L}(\Phi^{-1}(Z_k)),\mu) &\leq& d_W(\mathcal{L}(\Phi^{-1}(Z_k)),\widehat{\mu}_{\eta}\circ \Phi)+
d_W(\mu,\widehat{\mu}_{\eta}\circ \Phi),
\end{eqnarray*}
and that
$d_W(\mu,\widehat{\mu}_{\eta}\circ \Phi)
\leq C\eta^{\frac{\gamma}{2}},$
thus, we just need to give the upper bound for $d_W(\mathcal{L}(\Phi^{-1}(Z_k)),\widehat{\mu}_{\eta}\circ \Phi)$. It follows from definition of $d_W$ and the estimate for $\|\nabla \Phi^{-1}\|_{\infty}$ in \eqref{e:phi-infty} that
\begin{eqnarray*}
&&d_W(\mathcal{L}(\Phi^{-1}(Z_k)),\widehat{\mu}_{\eta}\circ \Phi) \ \ = \ \  \sup_{h\in {\rm Lip}(1)} \{ \E h(\Phi^{-1}(Z_k)) - \int_{\R^d} h(\Phi^{-1}(x)) \widehat{\mu}_{\eta}(dx) \} \\
&=& \| \nabla \Phi^{-1} \|_{\infty} \sup_{h\in {\rm Lip}(1)} \{ \E [\frac{1}{\| \nabla \Phi^{-1} \|_{\infty} } h(\Phi^{-1}(Z_k))] - \int_{\R^d} \frac{1}{\| \nabla \Phi^{-1} \|_{\infty} } h(\Phi^{-1}(x)) \widehat{\mu}_{\eta}(dx) \} \\
&=& \| \nabla \Phi^{-1} \|_{\infty} \sup_{h\in {\rm Lip}(1)} \{ \E h(Z_k) - \int_{\R^d} h(x) \widehat{\mu}_{\eta}(dx) \} \\
&=& \| \nabla \Phi^{-1}\|_{\infty} d_W(\mathcal{L}(Z_k),\widehat{\mu}_{\eta}) \\
&\leq& C\eta^{-1} e^{-ck\eta},
\end{eqnarray*}
where we used Lemma \ref{pro:ergodic-MC} for the last inequality and the positive constant $C$ is independent of $\eta$. Thus, we know
\begin{eqnarray}\label{e:er-st}
d_W(\mathcal{L}(\Phi^{-1}(Z_k)),\mu) &\leq& C \eta^{-1}e^{-ck\eta}+ C(\gamma) \eta^{\frac{\gamma}{2}}.
\end{eqnarray}

For any given error $\varepsilon>0$, taking $\eta \asymp \varepsilon^{\frac{2}{\gamma}}$ and $k\asymp \varepsilon^{-\frac{2}{\gamma}}|\log \varepsilon|$, we have
\begin{eqnarray*}
d_W(\mathcal{L}(\Phi^{-1}(Z_k)),\mu)
&\leq& d_W(\mathcal{L}(\Phi^{-1}(Z_k)),\widehat{\mu}_{\eta}\circ \Phi)+
d_W(\widehat{\mu}_{\eta}\circ \Phi,\mu) \\
&\leq& C\eta^{-1}e^{-ck\eta} + C(\gamma) \eta^{\frac{\gamma}{2}}
\ \ \leq \ \ \varepsilon.
\end{eqnarray*}
We obtain the desired result.
\end{proof}

\subsection{Proof of Theorem \ref{thm:main} (ii)}

In order to prove Theorem \ref{thm:main} (ii), we need the following proposition whose proof is postponed in Section \ref{sec:Jf-Md2}.

\begin{proposition}[Regularities, $b_1 \in \mathcal{C}_b^{\alpha}(\mathbb{R}^d)$ ]\label{pro:regu2}
Let $f$ be the solution to Eq. \eqref{e:steinequ}, under Assumption {\bf (A1) Case 2} and {\bf (A2)}, there exists some positive constant $C=C(\alpha,d)$ such that
\begin{eqnarray*}
|\nabla f(x)| &\leq& C(1+|x|^3), \\
|\nabla^2 f(x)| &\leq& C (1+|x|^4), \\
\sup_{y:|y-x|\leq 1/8 } \frac{|\nabla^2 f(x)-\nabla^2 f(y)|}{ |x-y| | \log |x-y| | } &\leq& C(1+|x|^{\frac{2}{\alpha}+1}).
\end{eqnarray*}
\end{proposition}

\begin{proof}[Proof of Theorem \ref{thm:main} (ii).]
For $h\in {\rm Lip}(1)$, let $f$ be the solution to  the Poisson equation \eqref{e:steinequ}. With the same notations as in the proof of Theorem \ref{thm:main} (i), we have
\begin{eqnarray*}
\mathbb{E}[\widehat{\mathcal{A}}f(Z_0) ]&=&\frac \eta 2 {\rm I}+\frac{1}{\eta} {\rm II},
\end{eqnarray*}
where
\begin{eqnarray*}
{\rm I} &=& -\mathbb{E} [ \langle \widehat{b}(Z_0) \widehat{b}^{\prime}(Z_0), \nabla^2 f(Z_0) \rangle_{\textrm{HS}} ], \\
{\rm II} &=& -\mathbb{E} \int_0^1 \int_0^1 r \langle \delta \delta^{\prime}, \nabla^2 f(Z_0+\tl{r} r \delta )-\nabla^2 f(Z_0) \rangle_{\textrm{HS}} \dif\tl{r} \dif r.
\end{eqnarray*}
Under assumption {\bf (A1) Case 2} and {\bf (A2)}, by Proposition \ref{pro:regu2}, with the similar calculations for claims \eqref{e:claim-1} and \eqref{e:claim-2}, we can show that there exists constant $C>0$ independent of $\eta$ such that
\begin{eqnarray*}
|{\rm I}| &\leq& C,  \\
|{\rm II}| &\leq& C\eta^{\frac{3}{2}}|\log \eta|,
\end{eqnarray*}
which implies that
\begin{eqnarray*}
\sup_{h\in {\rm Lip}(1)} |\mathbb{E}[\widehat{\mathcal{A}}f(Z_0)]|
\ \ \leq \ \  \sup_{h\in {\rm Lip}(1)} \{ \frac{\eta}{2} |{\rm I}|+\frac{1}{\eta} |{\rm II}| \}
\ \  \leq \ \ C\eta^{\frac{1}{2}}|\log \eta|.
\end{eqnarray*}
Combining this with Eq. \eqref{e:steinequ},  there exists constant $C>0$ independent of $\eta$ such that
\begin{eqnarray*}
d_W(\widehat{\mu},\widehat{\mu}_{\eta})
\ \ = \ \  \sup_{h\in {\rm Lip(1)} } \{  \widehat{\mu} (h) - \widehat{\mu}_{\eta} (h) \}
\ \ = \ \  \sup_{h\in {\rm Lip(1)} } \{ \mathbb{E}^{\widehat{\mu}_{\eta}}[\widehat{\mathcal{A}}f(Z_0)]  \}
\ \ \leq \ \ C\eta^{\frac{1}{2}} |\log \eta|.
\end{eqnarray*}

By the  similar calculations as for inequality \eqref{e:er-st}, one has
\begin{eqnarray*}
d_W(\mathcal{L}(\Phi^{-1}(Z_k)),\mu) &\leq& C\eta^{-1}e^{-ck\eta}+ C \eta^{\frac{1}{2}} |\log \eta|.
\end{eqnarray*}
For any given error $\varepsilon>0$, taking $\eta \asymp \varepsilon^{\frac{8}{3}}$ and $k\asymp \varepsilon^{-\frac{8}{3}}|\log \varepsilon|$, we obtain
\begin{eqnarray*}
d_W(\mathcal{L}(\Phi^{-1}(Z_k)),\mu)
&\leq& d_W(\mathcal{L}(\Phi^{-1}(Z_k)),\widehat{\mu}_{\eta}\circ \Phi)+
d_W(\widehat{\mu}_{\eta}\circ \Phi,\mu) \\
&\leq& C\eta^{-1}e^{-ck\eta} + C\eta^{\frac{1}{2}}|\log \eta|
\ \ \leq \ \ \varepsilon.
\end{eqnarray*}
The proof is complete.
\end{proof}

\section{Proof of Proposition \ref{pro:regu1}.}  \label{sec:Calpha}

Before proving Proposition \ref{pro:regu1}, we first give some definitions.

For  a real number $\tau$ and $k=0,1,2$, $\alpha\in (0,1]$, $\mathcal{D}$ is a open subset of $\mathbb{R}^d$, we define
\begin{align*}
[f]^{(\tau)}_{k,0;\mathcal{D}} &=[f]^{(\tau)}_{k;\mathcal{D}}= \sup_{x \in \mathcal{D}}  d^{k+\tau}_{x} |\nabla^k f(x)|, \\
[f]^{(\tau)}_{k,\alpha;\mathcal{D}} &= \sup_{x,y \in \mathcal{D}, x\neq y} d^{k+\alpha+\tau}_{x,y} \frac{ |\nabla^k f(x)- \nabla^k f(y)|  }{ |x-y|^{\alpha} }, \\
|f|^{(\tau)}_{k;\mathcal{D}}&=\sum_{j=1}^{k} [f]^{(\tau)}_{j;\mathcal{D}},\\
|f|^{(\tau)}_{k,\alpha;\mathcal{D}} &= |f|^{(\tau)}_{k;\mathcal{D}}+[f]^{(\tau)}_{k,\alpha;\mathcal{D}},
\end{align*}
where $d_x={\rm dist}(x,\partial \mathcal{D}),d_{x,y}=\min(d_x,d_y)$.
When $\tau=0$, we also write  $[f]^{*}_{k;\mathcal{D}}:=[f]^{(0)}_{k;\mathcal{D}}$, $[f]^{*}_{k,\alpha;\mathcal{D}}:=[f]^{(0)}_{k,\alpha;\mathcal{D}}$ and $|f|^{*}_{k,\alpha;\mathcal{D}}:=|f|^{(0)}_{k,\alpha;\mathcal{D}}$.
It follows from Gilbarg and Trudinger \cite[(6.11)]{gilbarg1977schauder}  that
\begin{equation}\label{e:product-inter-est}
|fg|^{(\tau_1+\tau_2)}_{0,\alpha;\mathcal{D}} \leq |f|^{(\tau_1)}_{0,\alpha;\mathcal{D}} |g|^{(\tau_2)}_{0,\alpha;\mathcal{D}} \quad \text{for} \quad \tau_1+\tau_2 \geq 0,
\end{equation}
if we denote $d_{\mathcal{D}}={\rm diam}(\mathcal{D})$, one has
\begin{eqnarray}\label{alpha-upperbound}
[u]^*_{k,\alpha;\mathcal{D}} \leq \max(1,d_{\mathcal{D}}^{k+\alpha})  [u]_{k,\alpha;\mathcal{D}}.
\end{eqnarray}
Let $\mathcal{D}^{\prime} \subset \subset \mathcal{D}$ and $d_{\mathcal{D}}^{\prime}={\rm dist}(\mathcal{D^{\prime}},\partial \mathcal{D})$, then we know
\begin{eqnarray}\label{alpha-lowerbound}
\min (1,(d_{\mathcal{D}}^\prime)^{k+\alpha}) [u]_{k,\alpha;\mathcal{D}^{\prime}} \leq [u]^*_{k,\alpha;\mathcal{D}}.
\end{eqnarray}

The following is a classic result from Gilbarg and Trudinger \cite[Theorem 6.2]{gilbarg1977schauder}.

\begin{lemma}[Schauder interior estimates for $\mathcal{C}^\alpha$ function]\label{lem:sch-est}
Let $\mathcal{D}$ be a open subset of $\mathbb{R}^d$, and let $u\in \mathcal{C}_{b}^{2,\alpha}(\mathcal{D})$ be a solution in $\mathcal{D}$ of the equation
\begin{equation}
\mathcal{L}u=\langle a(x), \nabla^2 u(x) \rangle_{\rm HS}+\langle  b(x), \nabla u(x)  \rangle=f,
\end{equation}
where $f\in \mathcal{C}^{\alpha}(\mathcal{D})$ and there are positive constants $\kappa,K$ such that
\begin{equation}\label{schauder-coff-est-1}
a^{ij} \xi_{i} \xi_{j} \geq \kappa |\xi|^2, \quad \forall x\in \mathcal{D},\xi \in \mathbb{R}^d,
\end{equation}
and
\begin{equation}\label{schauder-coff-est-2}
|a^{ij}|^{(0)}_{0,\alpha;\mathcal{D}},|b^i|_{0,\alpha;\mathcal{D}}^{(1)}\leq K.
\end{equation}
Then there exists some constant $C=C(d,\alpha,\kappa,K)$ such that
\begin{equation}\label{e:ustar}
|u|^{*}_{2,\alpha;\mathcal{D}} \leq C\left(|u|_{0;\mathcal{D}}+|f|^{(2)}_{0,\alpha;\mathcal{D}}\right).
\end{equation}
\end{lemma}

\begin{proof}[Proof of Proposition \ref{pro:regu1}]
 For any $x\in \mathbb{R}^d$ and $0<r(x)=\frac{1}{2(1+|x|)}\leq 1/2$, denote
\begin{eqnarray*}
\mathcal{B}_x=B_{r(x)}(x),
\end{eqnarray*} and consider the Poisson equaion \eqref{e:steinequ} on $\mathcal{B}_x$ as follows:
\begin{equation}\label{e:steinequ-sec5}
\widehat{\mathcal{A}} f=\langle  \widehat{b}(z), \nabla f(z)  \rangle + \frac{1}{2}  \langle \widehat{a}(z), \nabla^2 f(z) \rangle_{\rm HS}=h(z) - \widehat{\mu}(h).
\end{equation}

Recall that
\begin{equation}
\widehat{b}=(\lambda u+\nabla \Phi \cdot b_2) \circ \Phi^{-1}.
\end{equation}
By Corollary \ref{cor:hatb}, we have $\widehat{b}\in \mathcal{C}^{\gamma}(\mathbb{R}^d)$.  Combing this with the linear growth condition for $\widehat{b}$ in Lemma \ref{lem:tlb} below, there exists a constant $C(\lambda)$ such that
\begin{eqnarray*}
|\widehat{b}|_{0,\gamma;\mathcal{B}_x}^{(1)} &=& |\widehat{b}|^{(1)}_{0;\mathcal{B}_x}+[\widehat{b}]^{(1)}_{0,\gamma;\mathcal{B}_x}
\ \ \leq \ \   \sup_{y \in \mathcal{B}_x}  d_{y} |\widehat{b}(y)|+\sup_{y,z \in \mathcal{B}_x} d^{1+\gamma}_{y,z} \frac{ |\widehat{b}(y)-\widehat{b}(z)|  }{ |y-z|^{\gamma} }
 \ \ \leq \ \  C(\lambda).
\end{eqnarray*}

Meanwhile, for the transformed diffusion term, we have
\begin{eqnarray*}
\widehat{\sigma} &=& (\nabla \Phi \cdot \sigma) \circ \Phi^{-1},
\end{eqnarray*}
thus $\widehat{a}=\widehat{\sigma} \widehat{\sigma}^{\prime}\in \mathcal{C}^{\gamma}(\R^d)$. Moreover, it follows from Lemma \ref{lem:tlb} below that $\widehat{a}$ is strictly elliptic.

Taking $\mathcal{D}=\mathcal{B}_x$, the coefficients of Eq. \eqref{e:steinequ-sec5} satisfy the conditions \eqref{schauder-coff-est-1} and \eqref{schauder-coff-est-2} in Lemma \ref{lem:sch-est}, and thus we have
\begin{equation}\label{e:schauder-stein-1}
|f|^{*}_{2,\alpha;\mathcal{B}_x} \leq C\left(|f|_{0;\mathcal{B}_x}+|h-\widehat{\mu}(h)|^{(2)}_{0,\alpha;\mathcal{B}_x} \right),
\end{equation}
namely,
\begin{align}\label{alpha-rx-est}
|\nabla f|_{0;\mathcal{B}_x} &\leq C \left(|f|_{0;\mathcal{B}_{x}}+|h-\widehat{\mu}(h)|^{(2)}_{0,\gamma;\mathcal{B}_{x}}\right) (1+|x|),\nonumber \\
|\nabla^2 f|_{0;\mathcal{B}_x} &\leq C \left(|f|_{0;\mathcal{B}_{x}}+|h-\widehat{\mu}(h)|^{(2)}_{0,\gamma;\mathcal{B}_{x}}\right) (1+|x|)^{2},\nonumber\\
[f]_{2,\gamma;\mathcal{B}_x} &\leq  C  \left(|f|_{0;\mathcal{B}_{x}}+|h-\widehat{\mu}(h)|^{(2)}_{0,\gamma;\mathcal{B}_{x}}\right) (1+|x|)^{2+\gamma}.
\end{align}

For any fixed $0<r_0\leq 1$, we know
\begin{eqnarray*}
B_{r_0}(x)\subset \bigcup_{y\in B_{r_0}(x)} \mathcal{B}_y,
\end{eqnarray*}
thus
\begin{equation}
[f]_{2,\gamma;B_{r_0}(x)}\ \  \leq  \ \  C\sup_{y\in  B_{r_0}(x)}  (|f|_{0;\mathcal{B}_{y}}+|h-\widehat{\mu}(h)|^{(2)}_{0, \gamma; \mathcal{B}_{y}})(1+|x|)^{2+\gamma}.
\end{equation}

Recalling from Lemma \ref{lem:stein-f} that $|f(x)|\leq C(1+|x|^2)$ for all $x\in \R^d$,
 one has for all $x\in \mathbb{R}^d$,
\begin{equation}
\sup_{y\in  B_{r_0}(x)}  |f|_{0;\mathcal{B}_{y}} \leq \sup_{y\in  B_{r_0}(x)} \sup_{z\in \mathcal{B}_y} C(1+|z|^2) \leq \sup_{y\in  B_{r_0}(x)} 3 C(1+|y|^2)\leq 9 C(1+|x|^2),
\end{equation}
where we used the fact that
\begin{eqnarray*}
\sup_{y:|y|\leq 1} \left\{ \frac{1+|x+y|^2}{1+|x|^2} \right\} &\leq& 3.
\end{eqnarray*}

Since $|h(x)|\leq C(1+|x|^2)$ , we have $|h|^{(2)}_{0,1;\mathcal{B}_x}\leq C(1+|x|^2)$ and
\begin{eqnarray*}
\sup_{y\in  B_{r_0}(z)} |h|^{(2)}_{0,\gamma;\mathcal{B}_{y}} \leq  C (1+|x|^2).
\end{eqnarray*}
Let
\begin{eqnarray*}
H(x) &=& C(1+|x|^2)(1+|x|)^{2+\gamma}.
\end{eqnarray*}
We have, for any $x\in \mathbb{R}^d$,
\begin{align*}
|\nabla f(x)|\leq |\nabla f|_{0;\mathcal{B}_x}& \leq H(x)/(1+|x|)^{1+\gamma}, \\
|\nabla^2 f(x)|\leq |\nabla^2 f|_{0;\mathcal{B}_x}& \leq H(x)/(1+|x|)^{\gamma}, \\
[f]_{2,\gamma;B_{r_0}(x)} \leq& H(x).
\end{align*}
The proof is complete.
\end{proof}

\section{Proof of Proposition \ref{pro:regu2}.} \label{sec:Jf-Md2}
Before proving Proposition \ref{pro:regu2}, we give some definitions.
\begin{definition}[log-Lipschitz] \label{def:log-Lip}
For a function $f$, we define its log-Lipschitz semi-norm by
\begin{eqnarray}\label{e:log-Lip}
[f]_{\textrm{LL}} &=& \sup_{0<|x-y|<1} \frac{|f(x)-f(y)|}{|x-y||\log |x-y||},
\end{eqnarray}
and we call the function $f$ is log-Lipschitz if $[f]_{\textrm{LL}}<\infty$.
\end{definition}

Define the related norms on open set $\mathcal{D}$, for integers $k=0,1,2$, by
\begin{eqnarray*}
|f|_{k,{\rm LL};\mathcal{D}} &=& |f|_{k;\mathcal{D}}+\sup_{x,y \in \mathcal{D}, 0<|x-y|<1} \frac{ |\nabla^k f(x)- \nabla^k f(y)|  }{ |x-y||\log |x-y||}.
\end{eqnarray*}

Similarly, we define $[u]^{*}_{k,{\rm LL};\mathcal{D}}$ as
\begin{equation}\label{star-dis}
[u]^{*}_{k,{\rm LL};\mathcal{D}}=\sup_{x,y \in \mathcal{D}, 0<|x-y|<1} d^{k+1}_{x,y} \frac{ |\nabla^k f(x)- \nabla^k f(y)|  }{ |x-y||\log |x-y||},
\end{equation}
where $d_x={\rm dist}(x,\partial \mathcal{D}),d_{x,y}=\min(d_x,d_y)$. Denote $d_{\mathcal{D}}={\rm diam}(\mathcal{D})$. One has
\begin{eqnarray}\label{upperbound}
 [u]^*_{k,\textrm{LL};\mathcal{D}} &\leq& \max(1,d_{\mathcal{D}}^{k+1})  [u]_{k,\textrm{LL};\mathcal{D}}.
\end{eqnarray}
Let $\mathcal{D}^{\prime} \subset \subset \mathcal{D}$ and $d_{\mathcal{D}}^{\prime}={\rm dist}(\mathcal{D^{\prime}},\partial \mathcal{D})$, then
\begin{eqnarray}\label{lowerbound}
\min (1,(d_{\mathcal{D}}^\prime)^{k+1}) [u]_{k,\textrm{LL};\mathcal{D}^{\prime}} &\leq& [u]^*_{k,\textrm{LL};\mathcal{D}}.
\end{eqnarray}

We are at the position to state the following important Lemma which is crucial for proving Proposition \ref{pro:regu2}.
\begin{lemma}\label{pro:local}
For any $z\in \R^d$, let $f$ be a classical solution of Eq. \eqref{e:steinequ} in $B_{\rho}(z)$ with $0< \rho \leq 1/2$.
 If $h$ is Lipschitz, for $\alpha\in (0,1)$ there exists a positive constant $C$, independent of $z$, for $x, y\in B_{\rho/4}(z)$ such that
\begin{eqnarray}\label{e:uLip}
\frac{ |\nabla^2 f(x)-\nabla^2 f(y)|}{ |x-y| | \log |x-y| |  } &\leq &
C \left(| f |_{0;B_{\rho}(z)}+1\right)(1+|z|^{\frac{2}{\alpha}+1}).
\end{eqnarray}
\end{lemma}
The proof will given after the proof of Proposition \ref{pro:regu2}.

\begin{proof}[Proof of Proposition \ref{pro:regu2}]

From the inequality \eqref{e:claim-3} below, we know that $\widehat{b}$ is locally Lipschitz. On the other hand, one can  verify that $\widehat{a}$ is a Lipschitz function because of  the higher regularity of $b_1$ ($b_1\in \mathcal{C}_b^{\alpha}(\mathbb{R}^d)$).
Moverover, $\widehat{a}$ is strictly elliptic. Meanwhile, we can regard functions $\widehat{a}$, $\widehat{b}$ and $h$ as $\mathcal{C}^{\alpha}$ functions on the varying radius balls $\mathcal{B}_x$. Utilizing the same arguments as that in proof of Proposition \ref{pro:regu1}, we can show that  for any $x\in \mathbb{R}^d$,
\begin{align*}
|\nabla f(x)|\leq |\nabla f|_{0;\mathcal{B}_x}& \leq (1+|x|^3), \\
|\nabla^2 f(x)|\leq |\nabla^2 f|_{0;\mathcal{B}_x}& \leq (1+|x|^4).
\end{align*}

Combining with the regularity $\frac{|\nabla^2 f(x)-\nabla^2 f(y)|}{ |x-y| | \log |x-y| | }$ presented in Lemma \ref{pro:local} and the estimate for $|f(x)|$ in Lemma \ref{lem:stein-f}, that is, $|f(x)| \leq C (1+|x|^2)$ for all $x\in \R^d$, one obtains
\begin{equation*}
\sup_{y:|y-x|\leq 1/8} \frac{|\nabla^2 f(x)-\nabla^2 f(y)|}{ |x-y| | \log |x-y| | } \leq C(1+|x|^{\frac{2}{\alpha}+1}).
\end{equation*}
The proof is  complete.
\end{proof}
In order to prove Lemma \ref{pro:local}, we first consider the following simpler equation: for any fixed point $z\in \R^d$,
\begin{eqnarray} \label{e:ModelE}
\bar{\mathcal{L}} f(x) &=& \frac{1}{2}\langle \widehat{a}(x), \nabla^2 f(x) \rangle_{\rm HS} \ \ = \ \ \bar{h}, \quad x\in  B_{\rho} (z).
\end{eqnarray}
The regularity result of Eq.  \eqref{e:ModelE} that we need is stated in the following lemma, whose proof is standard, see for instance Wang \cite[Section 2]{wang2006Schauder}.
\begin{lemma}\label{freeze-cof}
Suppose that $\bar{h}$ in Eq. \eqref{e:ModelE} is  locally Lipschitz. Let  $f\in \mathcal{C}^{2}(B_{\rho}(z))$ be a classical solution to Eq. \eqref{e:ModelE}. Then there exists positive constant $C$, independent of $z$, for all $x,y\in B_{\rho/4}(z)$,
\begin{eqnarray}\label{freeze-cof-ieq}
\frac{ | \nabla^2 f(x) - \nabla^2 f(y) | }{ |x-y| | \log |x-y| | } &\leq&  C \left( | \nabla^2 f|_{0;B_\rho(z)} + | f |_{0;B_\rho(z)} + |\bar{h}|_{0,1;B_{\rho}(z)}\right).  
\end{eqnarray}
\end{lemma}
\vskip 0.4cm
\begin{proof}[ Proof of Lemma \ref{pro:local}.  ]
Recall the Poisson equation \eqref{e:steinequ}, that is,
\begin{eqnarray*}
\langle  \widehat{b}(x), \nabla f(x)  \rangle + \frac{1}{2}  \langle \widehat{a}(x), \nabla^2 f(x) \rangle_{\rm HS}  &=& h(x) - \widehat{\mu}(h).
\end{eqnarray*}
We can rewrite this equality as
\begin{eqnarray*}
\bar{\mathcal{L}} f(x) &=& h(x) - \widehat{\mu}(h) - \langle  \widehat{b}(x), \nabla f(x)  \rangle \ \ =: \ \  \bar{h}(x),
\end{eqnarray*}
where $\bar{\mathcal{L}}$ is defined  in Eq. \eqref{e:ModelE}.

We claim that for any $z\in \R^d$, there exists some positive constant $C$, independent of $z$ such that
\begin{eqnarray}\label{e:claim-3}
|\widehat{b}(x)|_{0,1;B_{\rho}(z)}
&\leq&  C (1+|z|).
\end{eqnarray}
The proof of this claim is given later. Assume for claim \eqref{e:claim-3} holds.
It follows from Lemma \ref{freeze-cof} that there exists positive constants $C$, independent of $z$ such that, for any $x,y\in B_{\rho/4}(z)$,
\begin{align}\label{e:D2-1}
& \frac{ |\nabla^2 f(x) - \nabla^2 f(y) | }{ |x-y| | \log |x-y| | }
\leq   C \left( | \nabla^2 f|_{0;B_\rho(z)} + | f |_{0;B_\rho(z)} + |\bar{h} (x)|_{0,1;B_{\rho}(z)}\right) \nonumber \\
\leq & C \left( | \nabla^2 f|_{0;B_\rho(z)} + | f |_{0;B_\rho(z)} + |h(x)|_{0,1;B_{\rho}(z)}+|\widehat{b}(x)|_{0,1;B_{\rho}(z)}|\nabla f(x)|_{0,1;B_{\rho}(z)}+1\right)\nonumber\\
\leq &  C (1+|z|) \left[ | \nabla^2 f|_{0;B_\rho(z)} + |f|_{0;B_{\rho}(z)}+ | \nabla f |_{0,1;B_\rho(z)}+1\right],
\end{align}
where the second inequality follows from the  inequality \eqref{e:product-inter-est} and the last inequality holds by the claim \eqref{e:claim-3}.

Using a similar interpolation inequality in Du and Liu \cite[Lemma 5.2]{DJ} (or Krylov \cite[Theorem 3.2.1]{krylov1996lectures}), we have for $\forall \varepsilon>0$ and $\alpha\in (0,1)$,
\begin{align*}
[f]_{0,1;B_{\rho}(z)} &\leq C \rho^{1+\alpha} \varepsilon^{1+\alpha} [f]_{2,\alpha;B_{\rho}(z)}+C \rho^{-1} \varepsilon^{-1} |f|_{0;B_\rho(z)}, \\
[f]_{2;B_{\rho}(z)} &\leq C \rho^\alpha \varepsilon^\alpha [f]_{2,\alpha;B_{\rho}(z)}+C \rho^{-2} \varepsilon^{-2} |f|_{0;B_\rho(z)},
\end{align*}
where $C$ depends on $\alpha, d$ and is independent of $z$. Observe that there exists some positive constant $C(\alpha,\rho)$ such that
\begin{eqnarray*}
[f]_{2,\alpha;B_\rho(z)} &\leq& C(\alpha,\rho) [\nabla^2 f]_{ {\rm LL};B_{\rho}(z)}.
\end{eqnarray*}
Thus, the right hand of inequality \eqref{e:D2-1} is smaller than
\begin{equation*}
C(\alpha,\rho) (1+|z|)\{ (1+\e \rho)(\e \rho)^{\alpha}[\nabla^2 f]_{ {\rm LL};B_{\rho}(z)} + (1+\e \rho) (\e \rho)^{-2}|f|_{0;B_{\rho}(z)}+|f|_{0;B_{\rho}(z)} +1 \},
\end{equation*}
while the constant $C(\alpha,\rho)$ is independent of $z$.

Taking $\varepsilon \rho=[4 C(\alpha,\rho)(1+|z|)]^{-1/\alpha} \wedge 1$, it follows from  \eqref{e:D2-1} that
\begin{equation*}
[\nabla^2 f]_{\textrm{LL};B_{\rho/4}(z)} \leq \frac{1}{2} [\nabla^2 f]_{\textrm{LL};B_\rho(z)}+C(1+|z|)^{1+\frac{2}{\alpha}} (1+|f|_{0;B_{\rho}(z)}).
\end{equation*}
Applying  Du and Liu \cite[Lemma 5.1]{DJ}, we get
\begin{equation*}
[\nabla^2 f]_{\textrm{LL};B_{\rho/4}(z)} \leq C(1+|z|^{\frac{2}{\alpha}+1}) (1+|f|_{0;B_{\rho}(z)} ),
\end{equation*}
which the desired inequality \eqref{e:uLip}. To complete the proof, it remains  to prove the claim \eqref{e:claim-3}. We first give some estimate for $[u]_{1,1;B_{\rho}(z)}$ for any $z\in \R^d$, {where $u$ is the solution of equation (\ref{e:pde-b1})} associated with the Zvonkin transform. By Lemma \ref{lem:regu-u} we have
\begin{align*}
[u]_{2,\alpha;B_{1/8}(z)} &\leq C(\lambda)( |b_1|_{0,\alpha;B_1(z)} +|u|_{0;B_1(z)}) , \\
\| u \|_{\infty} + \| \nabla u \|_{\infty} &\leq C(\lambda).
\end{align*}
Since $b_1\in \mathcal{C}_b^{\alpha}(\R^d)$, we have
\begin{eqnarray*}
|b_1|_{0,\alpha;B_1(z)} &=& |b_1|_{0;B_1(z)}+ [b_1]_{0,\alpha;B_1(z)}
\ \ \leq \ \ C.
\end{eqnarray*}
Thus we deduce that
\begin{align}\label{e:u2+}
[u]_{2,\alpha;B_{1/8}(z)} \leq C(\lambda), \ \
\| u \|_{\infty} + \| \nabla u \|_{\infty} \leq C(\lambda).
\end{align}

Then, by the interpolation inequality, we know that for any $\e>0$, there is some positive constant $C$, independent of $z$,  such that
\begin{align*}
[u]_{2;B_{1/8}(z)} &\leq C\e^{\alpha} [u]_{2,\alpha;B_{1/8}(z)} + C\e^{-2} |u|_{0;B_{1/8}(z)}.
\end{align*}
Taking $\e=1$ and combining the inequality \eqref{e:u2+}, we obtain
\begin{align*}
[u]_{2;B_{1/8}(z)} &\leq C(\lambda),
\end{align*}
where the constant $C$ is independent of $z$. Thus, for $\rho>1/8$
\begin{align}\label{e:u2}
[u]_{2;B_{\rho}(z)} \leq \sup_{y\in B_{\rho}(z)}
[u]_{2;B_{1/8}(y)} \leq C(\lambda).
\end{align}

Recall that
\begin{eqnarray*}
\widehat{b}(x)&=&(\lambda u+\nabla \Phi \cdot b_2) \circ \Phi^{-1}(x).
\end{eqnarray*}
It follows that
\begin{eqnarray*}
\nabla \widehat{b}(x)&=& \lambda \nabla u(\Phi^{-1}(x)) \nabla \Phi^{-1}(x)+[\nabla^2 \Phi \ b_2+\nabla \Phi \nabla b_2] \circ \Phi^{-1}(x)\cdot \nabla\Phi^{-1}(x).
\end{eqnarray*}
Since the term $b_2$ is Lipschitz, it follows from the estimates for $\| u \|_{\infty}$ and $\| \nabla u \|_{\infty}$ in Lemma \ref{lem:regu-u} and the inequality
\eqref{e:phi-infty} that there exists some positive constant $C(\lambda)$ such that
\begin{eqnarray*}
|\nabla \widehat{b}(x)|
&\leq& C(\lambda)+ C(\lambda)(1+|\Phi^{-1}(x)|) |\nabla^2 \Phi(\Phi^{-1}(x))| \nonumber \\
&\leq& C(\lambda)+ C(\lambda)(1+|\Phi^{-1}(x)|)     \\
&\leq& C(\lambda) (1+|x|),
\end{eqnarray*}
where the second inequality follows from the inequality \eqref{e:u2} and the last inequality holds because $\Phi^{-1}(x)=x-u(\Phi^{-1}(x))$. Combining the estimate for $|\nabla \widehat{ b}|$ above with the estimate for  $|\widehat{b}|$ in Lemma \ref{lem:tlb}, we proves the claim \eqref{e:claim-3}. The proof is complete.
\end{proof}

\begin{appendix}
\section{Ergodicity} \label{app:ergodicity}

\subsection{Exponential ergodicity for process $(Y_t)_{t\geq 0}$.}

Recall that the process $(Y_t)_{t\geq 0}$ in \eqref{e:SDE-1}, that is,
\begin{eqnarray*}
\dif Y_t &=& \widehat{b}(Y_t) \dif t + \widehat{\sigma}(Y_t) \dif W_t
\end{eqnarray*}
with initial value $Y_0=\Phi(X_0)$. By the Zvonkin's transform, we know that the new coefficients of the transformed equation  maintain the dissipative and linear growth conditions, see Xie and Zhang \cite[Proposition 7.8]{xie2017ergodicity} for more details.
\begin{lemma}\label{lem:tlb}
Under Assumption \ref{assump-1} (both for {\bf Case 1} and {\bf Case 2}), for sufficiently large $\lambda$ in Eq. \eqref{e:pde-b1}, one has

(i) there are $\widehat{\theta}_1$, $\widehat{\theta}_2$, $\widehat{\theta}_3>0$  such that for all $x \in \mathbb{R}^d$,
\begin{eqnarray*}
\langle \widehat{b}(x),x\rangle \leq -\widehat{\theta}_1 |x|^{2}+\widehat{\theta}_2,
\quad
|\widehat{b}(x)| \leq \widehat{\theta}_3(1+|x|).
\end{eqnarray*}

(ii) Let $\Lambda_1=\frac{1}{2} \lambda_{\sigma}$ and $\Lambda_2=2 \lambda^{-1}_{\sigma}$, one has
\begin{eqnarray*}
\Lambda_2 |\xi|^2 \geq \xi^{\prime} \widehat{\sigma}(x) \widehat{\sigma}^{\prime}(x) \xi \geq \Lambda_1 |\xi|^2, \quad \forall x\in \R^d, \xi \in \R^d.
\end{eqnarray*}
\end{lemma}

We will prove this lemma after the proof of Lemma \ref{pro:ergodic-Y}.

\begin{proof}[Proof of Lemma  \ref{pro:ergodic-Y}]

Let the Lyapunov function $V$ be given by
\begin{eqnarray}\label{e:V}
V(x) &=& 1+ |x|^2, \quad \forall x\in \R^d,
\end{eqnarray}
one has
\begin{eqnarray*}
\nabla V(x) \ \ = \ \ 2x, \quad
\nabla^2 V (x) \ \ = \ \ 2 {\rm I}_d.
\end{eqnarray*}
From expression of the operator $\widehat{\mathcal{A}}$ in \eqref{e:hatA}, by Lemma \ref{lem:tlb} and inequality \eqref{e:sigma-ell}, we see that  there exists a positive constant $c_1$ such that
\begin{eqnarray}\label{e:AV}
\widehat{\mathcal{A}} V(x) &=& 2 \langle  \widehat{b}(x), x \rangle +  \langle \widehat{\sigma}(x) \widehat{\sigma}^{\prime}(x),  {\rm I}_d \rangle_{\rm HS}  \nonumber \\
&\leq& -2\widehat{\theta}_1(|x|^2+1) +2\widehat{\theta}_1 + 2\widehat{\theta}_2 + d \| \nabla \Phi \|_{\infty}^2 \lambda^{-2}_{\sigma}  \nonumber \\
&=:& -2\widehat{\theta}_1(|x|^2+1) +c_2  \nonumber \\
&\leq& -\widehat{\theta}_1V(x)+c_1 1_A(x),
\end{eqnarray}
where the set $A=\{x:  |x|^2 \leq \frac{c_2}{\widehat{\theta}_1} -1\}$ and the constant $c_1=\widehat{\theta}_1 + 2\widehat{\theta}_2 + d \| \nabla \Phi \|_{\infty}^2 \lambda^{-2}_{\sigma}$. It follows from Meyn and Tweedie \cite[Theorem 6.1]{Meyn1993stability} that the process $(Y_t)_{t\geq 0}$ is exponentially ergodic and that there exist two positive constants $c$ and $C$ satisfying
\begin{eqnarray*}
\sup_{ |f| \leq 1+V }  \{ \widehat{P}_t f(y) - \widehat{\mu}(f)  \} &\leq& C V(y) e^{-ct}.
\end{eqnarray*}

We note that there exists some positive constant $C$ such that  for any $h\in {\rm Lip}_0(1)$,  $|h(x)| \leq C V(x)$ for all $x\in \R^d$. This implies that
\begin{eqnarray*}
\sup_{h \in {\rm Lip_0(1)}} \left \{ \widehat{P}_t h(y) - \widehat{\mu}(h) \right \} &\leq& C V(y) e^{-ct}.
\end{eqnarray*}
Combining this with equality \eqref{e:dW}, we obtain
$$d_W(\mathcal L(Y^y_t), \widehat \mu) \le  C (1+|y|^2) e^{-ct},  \forall \ y \in \R^d,$$
where $\mathcal L(Y^y_t)$ is the law of $Y^y_t$. The proof is complete.
\end{proof}


\begin{proof}[Proof of Lemma \ref{lem:tlb}. ]

We shall consider both {\bf Case 1} and {\bf Case 2} in Assumption \ref{assump-1}.

(i) Since $ \nabla \Phi(x) = {\rm I} + \nabla u(x)$, rewrite $\widehat{b}$ as
\begin{eqnarray*}
\widehat{b}(x) &=& (\lambda u+ ( {\rm I} + \nabla u ) \cdot b_2) \circ \Phi^{-1}(x).
\end{eqnarray*}

Recall that $\Phi(x) = x+ u(x)$, we know $x=\Phi^{-1}(x)+u\circ \Phi^{-1}(x)$ for any $x\in \R^d$. We first show the linear growth for $\widehat{b}$, that is,
\begin{eqnarray*}
|\widehat{b}(x)| &=& |\lambda u \circ \Phi^{-1}(x) + ( {\rm I} + \nabla u ) \cdot b_2) \circ \Phi^{-1}(x)| \\
&\leq& \lambda \| u \|_{\infty} + (1+\| \nabla u \|_{\infty} ) \theta_3 (1+ |\Phi^{-1}(x)| ) \\
&\leq& \widehat{\theta}_3(1+|x|),
\end{eqnarray*}
where the first inequality holds from the linear growth for $b_2$ in \eqref{e:lingro-b2}, and the last inequality holds from the fact that $\| u \|_{\infty}$ and $\| \nabla u \|_{\infty}$ are bounded for both {\bf Case 1} and {\bf Case 2}.

Next, we show the dissipation condition for $\langle x, \widehat{b}(x) \rangle$ by giving estimates for the corresponding three terms
$\langle x, \lambda u \circ \Phi^{-1}(x)  \rangle$, $\langle x, b_2  \circ \Phi^{-1}(x)   \rangle$ and $\langle x, (\nabla u \cdot b_2) \circ \Phi^{-1}(x) \rangle$.

Since $\| u \|_{\infty}$ is bounded, we have
\begin{eqnarray*}
|\langle x, \lambda u \circ \Phi^{-1}(x)  \rangle| &\leq& \lambda \| u \|_{\infty} |x|.
\end{eqnarray*}

Since $\| \nabla u \|_{\infty}$ is bounded  and since $b_2$ is of  the linear growth, one has
\begin{eqnarray*}
|\langle x,  (\nabla u \cdot b_2) \circ \Phi^{-1}(x) \rangle|
&\leq& \| \nabla u \|_{\infty} |x| |\theta_3|(1+ |\Phi^{-1}(x)|) \\
&=& |\theta_3|  \| \nabla u \|_{\infty} |x| (1+ |x-u\circ \Phi^{-1}(x) |) \\
&\leq& |\theta_3|  \| \nabla u \|_{\infty} |x| (1+ |x|+ \| u \|_{\infty}) \\
&\leq& C(\lambda) |x|^2 +  |\theta_3|  \| \nabla u \|_{\infty} (1+\| u \|_{\infty})|x|,
\end{eqnarray*}
where the last inequality holds from Lemmas \ref{lem:regu-ul} and \ref{lem:regu-u} and the constant $C(\lambda)$ depends on constant  $\lambda$.

For the last term, by the dissipation assumption in \eqref{e:dissi-b2} and linear growth condition in \eqref{e:lingro-b2} for $b_2$, we have
\begin{eqnarray*}
\langle x, b_2 \circ \Phi^{-1}(x) \rangle
&=& \langle \Phi^{-1}(x)+u\circ \Phi^{-1}(x), b_2 \circ \Phi^{-1}(x) \rangle  \\
&\leq& -\theta_1  |\Phi^{-1}(x)|^2 + \theta_2 + \| u \|_{\infty} \theta_3 (1+ |\Phi^{-1}(x)| ) \\
&=& -\theta_1  |x-u\circ \Phi^{-1}(x)|^2 + \theta_2 + \| u \|_{\infty} \theta_3 (1+ |x-u\circ \Phi^{-1}(x)| ) \\
&\leq& -\theta_1 |x|^2  + 2\theta_1 \| u \|_{\infty} |x| + \theta_2 + \| u \|_{\infty} \theta_3 (1+|x|+\| u \|_{\infty} ).
\end{eqnarray*}

Combining above three estimates, for large enough $\lambda>0$, we can see that  there exists some positive constants $\widehat{\theta}_1$ and $\widehat{\theta}_2$ such that
\begin{eqnarray*}
\langle \widehat{b}(x),x\rangle \leq -\widehat{\theta}_1 |x|^{2}+\widehat{\theta}_2, \ \ \forall x\in \R^d,
\end{eqnarray*}
proving the dissipation condition.

(ii) Recall that
\begin{eqnarray*}
\widehat{\sigma}(y) \ \ = \ \ (\nabla \Phi \cdot \sigma) \circ \Phi^{-1}(y),
\end{eqnarray*}
by the estimate for $\|\nabla \Phi\|_{\infty}$ in inequality \eqref{e:phi-infty} both for {\bf Case 1} and {\bf Case 2} in Assumption \ref{assump-1}, that is, $\frac{1}{2} \leq \|\nabla \Phi\|_{\infty} \leq 2$,
and the inequalities $\lambda_{\sigma} {\rm I}_d  \leq \sigma(x)\sigma^{\prime}(x) \leq \lambda_{\sigma}^{-1} {\rm I}_d$, it follows that
\begin{eqnarray*}
\Lambda_2 |\xi|^2 \geq \xi^{\prime} \widehat{\sigma}(y) \widehat{\sigma}^{\prime}(y) \xi \geq \Lambda_1 |\xi|^2, \quad \forall y\in \R^d, \xi \in \R^d
\end{eqnarray*}
with $\Lambda_1 =\frac{1}{2} \lambda_{\sigma}$ and $\Lambda_2 =2 \lambda^{-1}_{\sigma}$. The proof is complete.
\end{proof}

\subsection{Exponential ergodicity for process $(X_t)_{t\geq 0}$.}

%

\begin{proof}[Proof of Lemma  \ref{pro:ergodic-X}]
Let $\widehat{P}_t$ and $P_t$ denote the semigroups associated respectively with  $(Y_t)_{t\geq 0}$ and $(X_t)_{t\geq 0}$. Since $\Phi:\mathbb{R}^d \mapsto \mathbb{R}^d$ is a homeomorphism, it follows from Xie and Zhang \cite[Proposition 2.8]{xie2017ergodicity} that
\begin{eqnarray*}
\widehat{P}_t \varphi(y):=[P_t(\varphi \circ \Phi)] (\Phi^{-1}(y)),
\end{eqnarray*}
and  that the process $(X_t)_{t\geq 0}$ is also exponentially ergodic with ergodic measure $\mu$ given by
\begin{eqnarray*}
\mu & = & \widehat{\mu} \circ \Phi,
\end{eqnarray*}
where $\hat{\mu}$ is the unique invariant measure of $(Y_t)_{t\geq 0}$.
The proof is complete.
\end{proof}

\subsection{Exponential ergodicity of the  Markov chain $(Z_k)_{k\in \mathbb{N}_0}$.}
Recall that
\begin{eqnarray*}
Z_{k+1} &=& Z_{k}+ \eta \widehat{b}(Z_k)+  \sqrt{\eta} \widehat{\sigma}(Z_k)\xi_{k+1}, \quad \forall k\in \mathbb{N}_0,
\end{eqnarray*}
where $\eta>0$ is the step size and $\{ \xi_k \}_{k\in \mathbb{N}}$ are independent standard Gaussian random variables.

\begin{proof}[Proof of Lemma  \ref{pro:ergodic-MC}.]
(i) Let $\xi_1 \sim \mathcal{N}(0,{\rm I}_d)$, denote
\begin{eqnarray*}
\delta &=& Z_1^x - x \ \ = \ \ \eta \widehat{b}(x) + \eta^{\frac{1}{2}} \widehat{\sigma}(x) \xi_1.
\end{eqnarray*}

Let $V(x)=1+|x|^2$, one has
\begin{eqnarray*}
\mathbb{E} V (Z_1^x)- V(x)
&=& \mathbb{E}[ \langle \delta, \nabla V(x) \rangle  ]
+\mathbb{E} \int_0^1 \int_0^1 r \langle \delta \delta^{\prime}, \nabla^2 V (x+\tl{r} r \delta ) \rangle_{\textrm{HS}} \dif \tl{r} \dif r  \\
&=&  \eta \langle \widehat{b}(x),\nabla V(x) \rangle +\mathbb{E} \int_0^1 \int_0^1 r \langle \delta \delta^{\prime}, \nabla^2 V(x) \rangle_{\textrm{HS}} \dif \tl{r} \dif r \\
&=& \eta \widehat{\mathcal{A}}V(x) +\frac{\eta^2}{2} \langle \widehat{b}(x) \widehat{b}^{\prime} (x), \nabla^2 V(x) \rangle_{\textrm{HS}},
\end{eqnarray*}
where the second equality holds from $\nabla^2 V(x) = 2{\rm I}_d$ for all $x\in \R^d$. Taking into account the inequality \eqref{e:AV}, that is,
$
\widehat{\mathcal{A}} V(x) \leq -2\widehat{\theta}_1V(x) +c_2
$
for all $x\in \R^d$, there exists some positive constant $c_3>\widehat{\theta}_1$ independent of $\eta$ such that for sufficiently small  $\eta>0$,
\begin{eqnarray}\label{e:EVZ}
\mathbb{E} V (Z_1^x)
&\leq& V(x) + \eta(-2\widehat{\theta}_1 V(x) + c_2) + \frac{\eta^2}{2} \langle \widehat{b}(x) \widehat{b}^{\prime} (x), \nabla^2 V(x) \rangle_{\textrm{HS}}  \nonumber \\
&\leq& (1-\widehat{\theta}_1 \eta ) V(x) +c_3 \eta  \nonumber \\
&\leq& (1- \frac{\widehat{\theta}_1}{2} \eta ) V(x) +c_3 \eta 1_{B}(x),
\end{eqnarray}
where the second inequality holds because of  the linear growth of $\widehat{b}$ in Lemma \ref{lem:tlb}, and the set $B$ is defined by $B=\{x: |x|^2 \leq \frac{2c_3}{\widehat{\theta}_1} -1 \}$.

(ii) With similar calculations as in the proof of  Lemma \ref{pro:ergodic-Y}, we know that for any integer $\ell>1$, there exists positive constant $C_{\ell}$ {independent of $\eta$} such that,
\begin{eqnarray}\label{EVZL}
\widehat{\mathcal{A}} V^{\ell}(x) &\leq& - \ell \widehat{\theta}_1 V^{\ell}(x) + C_{\ell},  \quad \forall x\in \R^d.
\end{eqnarray}
{By the similar calculations in inequality \eqref{e:EVZ}, one gets from (\ref{EVZL}) that there exists some positive constant $\tl{C}_{\ell}$ independent of $\eta$ satisfying}
\begin{eqnarray}\label{e:EVellZ}
\mathbb{E} V^{\ell}(Z_1^x)
&\leq& (1- \frac{1}{2} \eta \ell \widehat{\theta}_1 ) V^{\ell}(x) +\tl{C}_{\ell} \eta.
\end{eqnarray}

For any integers $n\geq 1$, denote
\begin{eqnarray*}
V_n(x) &=& e^{\frac{1}{4}\widehat{\theta}_1 n \eta} V(x), \ \ \forall x\in \R^d, \ \
r(n) \ \ = \ \ \frac{\widehat{\theta}_1}{4}\eta e^{\frac{1}{4}\widehat{\theta}_1 n \eta},\ \
\bar {c} \ \ = \ \ \frac{4c_3}{\widehat{\theta}_1} e^{\frac{1}{4}\widehat{\theta}_1 \eta}.
\end{eqnarray*}
and the set
\begin{eqnarray*}
\mathscr{C} &=& \{x: V(x) \leq \frac{4c_3}{\widehat{\theta}_1} e^{\frac{1}{4}\widehat{\theta}_1 \eta} \}.
\end{eqnarray*}
Since $c_3>\widehat{\theta}_1$, we know the set $\mathscr{C}$ is non-empty and compact. By the inequality \eqref{e:EVZ}, one yields that
\begin{eqnarray*}
\widehat{Q}_1 V_{n+1}(x) + r(n) V(x)
&=& e^{\frac{1}{4}\widehat{\theta}_1 (n+1)\eta} \E V(Z_1^x) + r(n) V(x) \\
&\leq& e^{\frac{1}{4}\widehat{\theta}_1 (n+1)\eta}[ (1-\eta \widehat{\theta}_1) V(x) + c_3 \eta ] + r(n) V(x) \\
&=& V_n(x) + [(1-\eta \widehat{\theta}_1)e^{\frac{1}{4}\widehat{\theta}_1 \eta} + \frac{1}{4}\widehat{\theta}_1 \eta-1 ]e^{\frac{1}{4} \widehat{\theta}_1 n \eta} V(x) + c_3 \eta e^{\frac{1}{4}\widehat{\theta}_1 (n+1)\eta} \\
&\leq& V_n(x) + \bar{c} r(n) 1_{\mathscr{C}}(x),
\end{eqnarray*}
where the last inequality holds from the fact that for sufficiently small $\eta>0$,
\begin{equation*}
(1-\eta \widehat{\theta}_1)e^{\frac{1}{4}\widehat{\theta}_1 \eta} + \frac{1}{4}\widehat{\theta}_1 \eta-1 \leq -\frac{1}{4} \widehat{\theta}_1 \eta.
\end{equation*}

We claim that the compact set $\mathscr{C}$ is petite, then it follows from Tuominen and Tweedie \cite[Theorem 2.1]{Tuominen1994Subgeometric} or Douc et al. \cite[Theorem 1.1]{Douc2004practical} that
\begin{eqnarray*}
\sup_{ |f| \leq 1+|\cdot|^2 }  \{  \widehat{Q}_k f(x) - \widehat{\mu}_{\eta}(f)  \} &\leq& C \eta^{-1} e^{-ck\eta},
\end{eqnarray*}
where $\widehat{\mu}_{\eta}$ denotes the unique invariant measure of the Markov chain $(Z_k)_{k\in \mathbb{N}_0}$ and constants $C$ and $c$ are independent of $\eta$ and $k$. We denote that there exists some positive constant $C$ such that  for any $h\in {\rm Lip}_0(1)$,  $|h(x)| \leq C V(x)$ for all $x\in \R^d$.
Combining this with equality \eqref{e:dW}, we obtain
$$d_W(\mathcal L(Z^x_k), \widehat \mu_{\eta}) \le  C\eta^{-1} e^{-ck\eta},  \forall \ x \in \R^d,$$
where $\mathcal L(Z^x_k)$ is the law of $Z^x_k$.

The reminder is to show that the claim: the compact set $\mathscr{C}$ is petite. It is suffice to show that
\begin{eqnarray}\label{e:pet}
p(\eta,x,z) \geq c \nu(z), \ \ \forall x\in \mathscr{C},
\end{eqnarray}
where $p(\eta,x,z)$ is the density of $Z_1^x$, $c$ is some positive constant and $\nu$ is a probability measure. Recall that
\begin{eqnarray*}
Z_1^x &=& x+\eta \widehat{b}(x) + \eta^{\frac{1}{2}} \widehat{\sigma}(x) \xi_1,
\end{eqnarray*}
one gets the expression of $p(\eta,x,z)$ as below:
\begin{eqnarray*}
p(\eta,x,z)&=&[(2\pi)^d\eta^d \det(\widehat{\sigma}(x)\widehat{\sigma}^{\prime}(x) )]^{-\frac{1}{2}} \exp\left( -(z-x-\eta \widehat{b}(x)) \frac{(\widehat{\sigma}(x)\widehat{\sigma}^{\prime}(x))^{-1}}{2\eta}(z-x-\eta \widehat{b}(x))    \right).
\end{eqnarray*}

{It follows from Lemma \ref{lem:tlb} (ii) and  the fact that,}
$$
|z-x-\eta \widehat{b}(x)|^2 \ \ \leq \ \ 2|z|^2 + 4|x|^2 + 8\eta^2 \widehat{\theta}_3^2(1+|x|^2),
$$
we know $p(\eta,x,z)$ is larger than
\begin{eqnarray*}
(\pi\eta \Lambda_1)^{-\frac{d}{2}} \exp(-\frac{|z|^2}{\eta \Lambda_1})\times( \frac{2\Lambda_2}{\Lambda_1})^{-\frac{d}{2}} \exp(-\frac{1}{2\Lambda_1 \eta}[4|x|^2+8\eta^2\widehat{\theta}_3^2(1+|x|^2)]),
\end{eqnarray*}
thus the inequality \eqref{e:pet} holds by taking
\begin{eqnarray*}
\nu(z)&=&(\pi\eta \Lambda_1)^{-\frac{d}{2}} \exp(-\frac{|z|^2}{\eta \Lambda_1}),
\end{eqnarray*}
and
\begin{eqnarray*}
c&=&\inf_{x\in \mathscr{C}}\left\{( \frac{2\Lambda_2}{\Lambda_1})^{-\frac{d}{2}} \exp(-\frac{1}{2\Lambda_1 \eta}[4|x|^2+8\eta^2\widehat{\theta}_3^2(1+|x|^2)])\right\}>0
\end{eqnarray*}
for some non-empty and compact set $\mathscr{C}$.

(iii) Recall that the inequality \eqref{e:EVellZ}
\begin{eqnarray*}
\mathbb{E} V^{\ell}(Z_1^x)
&\leq& (1- \frac{1}{2} \eta \ell \widehat{\theta}_1 ) V^{\ell}(x) +\tl{C}_{\ell} \eta,
\end{eqnarray*}
one has
\begin{eqnarray*}
\widehat{\mu}_{\eta}(V^{\ell})
&\leq& (1- \frac{1}{2} \eta \ell \widehat{\theta}_1 ) \widehat{\mu}_{\eta}(V^{\ell}) +\tl{C}_{\ell} \eta,
\end{eqnarray*}
it implies that there exists some positive constant $C$ independent of $\eta$ such that
\begin{eqnarray*}
\widehat{\mu}_{\eta}(V^{\ell})
&\leq& \frac{2\tl{C}_{\ell}}{ \ell \widehat{\theta}_1}
\ \ \leq \ \ C.
\end{eqnarray*}
Using the relationship between $V$ and $|\cdot|^2$, we can get the desired result. The proof is complete.
\end{proof}

\end{appendix}

\end{document}